\documentclass[10pt]{article}
\usepackage[a4paper]{geometry}
\usepackage{graphicx}
\usepackage{amssymb}
\usepackage{amsmath}
\usepackage{booktabs}
\usepackage{paralist}
\usepackage{bm}
\usepackage{bbm}
\usepackage{cancel}
\usepackage{url}
\usepackage[usenames]{color}
\usepackage{mathtools}
\usepackage[only,llbracket,rrbracket]{stmaryrd}
\usepackage[normalem]{ulem} 
\usepackage{tikz}

\usepackage[textsize=tiny,color=blue!30!white,obeyFinal]{todonotes}
\usepackage{amsthm}

\usepackage[breaklinks,bookmarks=false]{hyperref}
\hypersetup{colorlinks, linkcolor=blue, citecolor=blue,
urlcolor=blue, plainpages=false, pdfwindowui=false,
pdfstartview={FitH}}

\numberwithin{equation}{section}

\def\cyee#1{{#1}}

\newcommand{\be}{\begin{equation}}
\newcommand{\ee}{\end{equation}}
\newcommand{\bea}{\begin{eqnarray}}
\newcommand{\eea}{\end{eqnarray}}
\newcommand{\bean}{\begin{eqnarray*}}
\newcommand{\eean}{\end{eqnarray*}}
\def\ba#1\ea{\begin{align}#1\end{align}}
\def\ban#1\ean{\begin{align*}#1\end{align*}}
\def\bat#1\eat{\begin{alignat}#1\end{alignat}}
\def\batn#1\eatn{\begin{alignat*}#1\end{alignat*}}
\def\bs#1\es{\begin{split}#1\end{split}}
\newcommand{\bse}{\begin{subequations}}
\newcommand{\ese}{\end{subequations}}
\newcommand{\bt}{\begin{theorem}}
\newcommand{\et}{\end{theorem}}
\newcommand{\bcj}{\begin{conjecture}}
\newcommand{\ecj}{\end{conjecture}}
\newcommand{\bl}{\begin{lemma}}
\newcommand{\el}{\end{lemma}}
\newcommand{\bc}{\begin{corollary}}
\newcommand{\ec}{\end{corollary}}
\newcommand{\bp}{\begin{proof}}
\newcommand{\ep}{\end{proof}}
\newcommand{\bd}{\begin{definition}}
\newcommand{\ed}{\end{definition}}
\newcommand{\br}{\begin{remark}}
\newcommand{\er}{\end{remark}}
\newcommand{\bas}{\begin{assumption}}
\newcommand{\eas}{\end{assumption}}
\newcommand{\bex}{\begin{example}}
\newcommand{\eex}{\end{example}}
\newcommand{\bqo}{\begin{quote}}
\newcommand{\eqo}{\end{quote}}
\newcommand{\bdc}{\begin{description}}
\newcommand{\edc}{\end{description}}
\newcommand{\bi}{\begin{itemize}}
\newcommand{\ei}{\end{itemize}}
\newcommand{\ben}{\begin{enumerate}}
\newcommand{\een}{\end{enumerate}}

\newtheorem{lemma}{Lemma}[section]
\newtheorem{proposition}[lemma]{Proposition}
\newtheorem{corollary}[lemma]{Corollary}
\newtheorem{assumption}[lemma]{Assumption}
\newtheorem{theorem}[lemma]{Theorem}
\newtheorem{definition}[lemma]{Definition}
\newtheorem{remark}[lemma]{Remark}
\newtheorem{example}[lemma]{Example}

\newtheorem{conjecture}[lemma]{Conjecture}

\newcommand\Om{\Omega}
\newcommand\om{\omega}
\newcommand\oma{{\omega_\ver}}

\newcommand\Gr{\nabla}

\newcommand\Dv{\nabla {\cdot}}

\newcommand\dv{\mathrm{div}}
\newcommand\scp{{\cdot}}

\newcommand\il{|\hspace{-0.02cm}|\hspace{-0.02cm}|}

\DeclarePairedDelimiter\enorm{\il}{\il}

\newcommand\Ho{H^1(\Om)}
\newcommand\Hoo{H^1_0(\Om)}

\newcommand\Lti[1]{L^2(#1)}

\newcommand\Hdv{\bm{H}(\dv,\Om)}

\newcommand\ie{i.e.}
\newcommand\cf{cf.}
\newcommand\eg{e.g.}
\newcommand\eal{{\em et al.}}
\newcommand\eq{\coloneqq}

\newcommand\pt{\partial}

\newcommand{\elm}{{K}}
\newcommand{\elmt}{{K'}}

\newcommand{\ver}{\cyee{\bm a}}

\newcommand{\Fhint}{\mathcal{F}_{\Omega}}

\newcommand\Th{\mathcal{T}}

\newcommand\Ta{{\mathcal{T}_{\ver}}}

\newcommand\Vhint{\mathcal{V}^{\mathrm{int}}}
\newcommand\Vhext{\mathcal{V}^{\mathrm{ext}}}

\newcommand\uu{u}
\newcommand\uh{u_{\cyee \idx}}
\newcommand\vh{v_{\cyee \idx}}

\newcommand\frh{{\bm \sigma}_\Th}

\newcommand{\eps}{\varepsilon}

\newcommand\tn{{\bm n}}

\newcommand\T{\mathcal{T}}

\newcommand\PP{{\mathbb P}}
\newcommand\RR{{\mathbb R}}

\DeclareMathOperator{\Div}{div}
\DeclareMathOperator{\Dim}{dim}
\DeclareMathOperator*{\argmin}{arg\,min}

\newcommand\ft{{\frac 1 2}}
\newcommand\mft{{-\frac 1 2}}

\newcommand{\lVERT}{\lvert\kern-0.25ex\lvert\kern-0.25ex\lvert}
\newcommand{\rVERT}{\rvert\kern-0.25ex\rvert\kern-0.25ex\rvert}
\newcommand{\NORM}[1]{{\lVERT#1\rVERT}}
\newcommand{\pair}[2]{\langle #1,#2 \rangle}
\newcommand{\norm}[1]{\lVert#1\rVert} 
\newcommand{\jump}[1]{\llbracket#1\rrbracket} 
\newcommand{\abs}[1]{\lvert#1\rvert} 

\newcommand{\tends}{\rightarrow}

\newcommand{\calF}{\mathcal{F}}
\newcommand{\frakF}{\mathfrak{F}}

\newcommand{\frakT}{\mathfrak{T}}

\newcommand{\calI}{\mathcal{I}}

\newcommand{\calR}{\mathcal{R}}

\newcommand{\calV}{\mathcal{V}}

\newcommand{\Ctr}{C_{\mathrm{Tr}}}

\newcommand{\dd}{\mathrm{d}}

\newcommand{\p}{\partial}
\newcommand{\Vh}{V_{\cyee \idx}}
\renewcommand{\dim}{d}

\newcommand{\sh}{\bm{\sigma}_{\cyee \idx}}

\newcommand{\RTN}{\bm{RTN}} 
\newcommand{\RTNp}{\RTN_{p}(\Th)}
\newcommand{\R}{\mathbb{R}}
\newcommand{\Vint}{\mathcal{V}^{\mathrm{int}}}
\newcommand{\Vext}{\mathcal{V}^{\mathrm{ext}}}
\newcommand{\Hdiv}{\bm{H}(\Div)}

\newcommand{\Hdiva}{\bm{H}(\Div,\oma)}

\newcommand{\RTNK}{\bm{RTN}_p(K)}
\newcommand{\RTNa}{\bm{RTN}_p(\Ta)}

\newcommand{\Qa}{Q_\idx^{\ver}}
\newcommand{\Va}{\bm{V}_\idx^\ver}

\newcommand{\bvh}{\bm{v}_{\cyee \idx}}
\newcommand{\bv}{\bm{v}}
\newcommand{\bn}{\bm{n}}

\newcommand{\Pih}{\Pi_{\cyee \idx}}
\newcommand{\calFint}{\Fhint}

\newcommand{\bL}{\bm{L}^2(\Om)}
\newcommand{\psia}{\psi_{\ver}}
\newcommand{\VK}{\mathcal{V}_K}
\newcommand{\phia}{\phi_{\cyee \idx}^{\ver}}
\newcommand{\ghia}{\gamma_\idx^{\ver}}
\newcommand{\sha}{\bm{\sigma}_{\cyee \idx}^{\ver}}

\newcommand{\woma}{w_{\ver}}
\newcommand{\wk}{w_{\elm}}
\newcommand{\wkt}{\widetilde{w}_{\elm}}
\newcommand{\ak}{\alpha_{\elm}}
\newcommand{\akp}{\alpha_{\elm^{\prime}}}
\newcommand{\af}{\alpha_F}

\newcommand{\btauh}{\bm{\tau}_{\cyee \idx}}
\newcommand{\bssub}{\bm{\sigma}_{\widetilde\idx}}
\newcommand{\calFa}{\calF_{\ver}}

\newcommand{\phih}{\phi_{\cyee \idx}}
\newcommand{\fh}{f_{\cyee \idx}}

\newcommand{\Cpo}{C_{p,1}}
\newcommand{\Cpd}{C_{p,\dim}}
\newcommand{\sD}{\{1{:}\dim\}}
\newcommand{\Qd}{Q^{\dim}}
\newcommand{\Kd}{K^\dim}

\newcommand{\Cpdp}{C_{p+1,\dim,\p K}}
\newcommand{\CpdK}{C_{p+1,\dim,K}}

\newcommand{\sreg}{\vartheta_{\T}}

\newif \ifT
\Ttrue

\ifT \newcommand{\idx}{\Th} \else \newcommand{\idx}{h} \fi

\title{Simple and robust equilibrated flux a posteriori estimates for singularly perturbed reaction--diffusion problems\thanks{This project has received funding from the European Research Council (ERC) under the European Union's Horizon 2020 research and innovation program (grant agreement No 647134 GATIPOR).}}

\author{Iain Smears\footnotemark[2] \and Martin Vohral\'ik\footnotemark[3] \footnotemark[4]}

\begin{document}
\maketitle

\renewcommand{\thefootnote}{\fnsymbol{footnote}}

\footnotetext[2]{Department of Mathematics, University College London, Gower
Street, WC1E 6BT London, United Kingdom
(\href{mailto:i.smears@ucl.ac.uk}{\texttt{i.smears@ucl.ac.uk}}).}

\footnotetext[3]{Inria, 2 rue Simone Iff, 75589 Paris, France (\href{mailto:martin.vohralik@inria.fr}{\texttt{martin.vohralik@inria.fr}}).}

\footnotetext[4]{Universit\'e Paris-Est, CERMICS (ENPC), 77455 Marne-la-Vall\'ee 2, France.}

\renewcommand{\thefootnote}{\arabic{footnote}}

\begin{abstract}
\noindent We consider energy norm a posteriori error analysis of conforming finite element approximations of singularly perturbed reaction--diffusion problems on simplicial meshes in arbitrary space dimension. Using an equilibrated flux reconstruction, the proposed estimator gives a guaranteed global upper bound on the error without unknown constants, and local efficiency robust with respect to the mesh size and singular perturbation parameters. Whereas previous works on equilibrated flux estimators only considered lowest-order finite element approximations and achieved robustness through the use of boundary-layer adapted submeshes or via combination with residual-based estimators, the present methodology applies in a simple way to arbitrary-order approximations and does not request any submesh or estimators combination. The equilibrated flux is obtained via local reaction--diffusion problems with suitable weights (cut-off factors), and the guaranteed upper bound features the same weights. We prove that the inclusion of these weights is not only sufficient but also necessary for robustness of any flux equilibration estimate that does not employ submeshes or estimators combination, which shows that some of the flux equilibrations proposed in the past cannot be robust. To achieve the fully computable upper bound, we derive explicit bounds for some inverse inequality constants on a simplex, which may be of independent interest.
\end{abstract}

\bigskip


\section{Introduction} \label{sec_intr}

Let $\Om$ be a polygonal/polyhedral/polytopal domain in $\RR^d$, $d\geq 1$, with a Lipschitz-continuous boundary.
Let $\eps >0$ and $\kappa \geq 0$ be two fixed real parameters, and let $f\in L^2(\Om)$ be a given source term.
Consider the problem: find $\uu: \Om \rightarrow \RR$ such that
\begin{subequations} \label{eq_RD} \begin{alignat}{2}
    - \eps^2 \Delta \uu + \kappa^2 \uu & = f \qquad & & \mbox{ in } \, \Om, \label{eq_RD_eq} \\
    \uu & = 0 & & \mbox{ on } \, \pt \Om. \label{eq_RD_BC}
\end{alignat} \end{subequations}
Let $a(\cdot,\cdot)$ be the symmetric bilinear form defined by
\be \label{eq:a}
    a(w,v)\eq \eps^2 (\nabla w,\nabla v) + \kappa^2 (w,v), \qquad w, v \in \Hoo,
\ee
where $({\cdot},{\cdot})$ denotes the $L^2$-inner product of scalar- and vector-valued functions on $\Omega$, with associated norm $\norm{\cdot}$.
The restriction of the $L^2$-inner product to an open subset $\omega\subset \Om$ is denoted by~$({\cdot},{\cdot})_{\om}$, with associated norm $\norm{{\cdot}}_{\om}$.
The weak formulation of problem~\eqref{eq_RD} is to find $u\in \Hoo$ such that
\begin{equation}\label{eq:RD_weak}
\begin{aligned}
a(u,v) = (f,v) &&& \forall\, v\in \Hoo.
\end{aligned}
\end{equation}
The energy norm $\NORM{{\cdot}}$ associated to problem~\eqref{eq_RD} is then the norm induced by the form $a({\cdot},{\cdot})$, namely
\be \label{eq:en_norm}
\begin{aligned}
    \NORM{v}^2\eq a(v,v), &&& v\in \Hoo.
\end{aligned}
\ee
In this paper, we shall be primarily interested in the case where $\eps \ll \kappa$, when problem~\eqref{eq_RD} is said to be {\em singularly perturbed}.
Then, the accurate numerical approximation can be challenging due to the typical presence of sharp boundary and/or interior layers in the solution.

In order to present more specifically the focus of this work, let us consider a simplicial mesh $\Th$ of~$\Om$ and let $\Vh\eq \PP_p(\Th)\cap \Hoo$ denote the subspace of $\Hoo$ of piecewise polynomial functions of degree at most $p$, where $p\geq 1$ is a fixed integer.
The conforming Galerkin finite element approximation of~\eqref{eq:RD_weak} consists of finding $\uh \in \Vh$ such that
\begin{equation}\label{eq:RD_fem}
a(\uh,\vh) = (f,\vh) \qquad \forall\,\vh\in\Vh.
\end{equation}
The goal is to find a computable a posteriori error estimator $\eta(\uh)$ that satisfies
\begin{equation}\label{eq_est_rel_eff}
\begin{aligned}
    \enorm{\uu - \uh} \leq C_{\mathrm{rel}} \eta(\uh), &&&  \eta(\uh) \leq C_{\mathrm{eff}} \enorm{\uu - \uh} + \text{data oscillation}.
    \end{aligned}
\end{equation}
The first inequality in~\eqref{eq_est_rel_eff} is called reliability, while the second inequality is called (global) efficiency. A localized version of the efficiency bound is actually desirable. The quality of the estimator is determined by the product of the two constants $C_{\mathrm{rel}}$ and $C_{\mathrm{eff}}$.
A key requirement for singularly perturbed problems is to obtain estimators that are {\em robust} in the sense that both constants $C_{\mathrm{rel}}$ and $C_{\mathrm{eff}}$ are independent of the singular perturbation parameters $\eps$ and $\kappa$. Only such estimates can quantify well the error in the numerical approximation and be reliably used in adaptive algorithms which allow for efficient approximation of the localized features of the solution.

Recently, several methodologies for constructing error estimators that satisfy~\eqref{eq_est_rel_eff} in a robust way have been studied.
Verf{\"u}rth~\cite{Verf_RD_rob_a_post_98} (see also~\cite{Verf_rob_a_post_CD_05} or~\cite[Section~4.3]{Verf_13}) was probably the first to show robust bounds, in the framework of the so-called residual-based estimates.
For the problem at hand, these estimators take the form (up to the data oscillation term and possible generic constants)
\begin{equation}\label{eq:residual_estimator}
\eta_{\mathrm{res}}(\uh)^2 \eq \sum_{\elm\in\Th} \ak^2 \norm{r_\idx}_{\elm}^2 + \sum_{F\in\calFint} \eps^{-1} \af \norm{j_\idx}_{F}^2,
\end{equation}
where the local element and face residuals are defined respectively by
\bse\label{eq:res}\ba
    r_\idx|_{\elm} & \eq (f+\eps^2 \Delta_\idx \uh - \kappa^2 \uh)|_{\elm}, \label{eq:elm_res}\\
    j_\idx|_{F} & \eq - \eps^2 \jump{ \nabla \uh \scp \bn_F }_F, \label{eq:face_res}
\ea \ese
and where $\Delta_\T$ denotes the element-wise Laplacian, $\jump{ \nabla \uh \scp \bn_F }_F$ denotes the jump of the normal component of $\nabla \uh$ over the face $F$, $\calFint$ stands for the set of internal faces of the mesh $\Th$, and the weights (cut-off factors) take the form
\be \label{eq:residual_weights}
    \alpha_S \eq \min\left\{\frac{h_S}{\eps},\frac{1}{\kappa}\right\},
\ee
with $h_S$ being the diameter of $S$, where $S$ is either a simplex $\elm$ or a face $F$.
The resulting estimator $\eta_{\mathrm{res}}(\uh)$ is thus a straightforward extension from the pure diffusion case $\kappa = 0$ and is simple to implement in practice.
The proof that $\eta_{\mathrm{res}}$ satisfies the second inequality in~\eqref{eq_est_rel_eff} rests on a bubble function technique, where the face bubble functions are defined with respect to a {\em submesh} matching the boundary-layer length scales and are possibly very steeply decaying. Their role is to capture the sharp layers caused by the singular perturbation.
Note that these bubble functions, and hence the submeshes on which they are defined, are only employed in the analysis; thus they do not need to be constructed in practice.
Shortly after, Ainsworth and Babu{\v{s}}ka~\cite{Ains_Bab_rel_rob_a_post_RD_99} extended the method of equilibrated residuals, \cf~\cite{Ainsw_Oden_a_post_FE_00}, to satisfy~\eqref{eq_est_rel_eff} in a robust way for lowest-order approximations, \ie\ $p=1$.
In contrast to the residual-based estimators, a boundary-layer adapted submesh in each mesh element needs to be {\em constructed in practice} in order to evaluate the estimator.

Further progress has been made since, although, to the best of our knowledge, only in the case of lowest-order approximations where the polynomial degree $p=1$.
Robust estimates that are guaranteed ($C_{\mathrm{rel}} = 1$) and where $\eta(\uh)$ is fully computable have been obtained in Cheddadi {\em et
al.}~\cite{Ched_Fuc_Priet_Voh_guar_rob_FE_RD_09}.
This remedies that $C_{\mathrm{rel}}$ is unknown for residual-based estimates and that exact solutions of some infinite-dimensional boundary value problems on each element (which cannot be performed exactly in practice) are required in the equilibrated residuals approach.
The estimator in~\cite{Ched_Fuc_Priet_Voh_guar_rob_FE_RD_09} is based on an equilibrated flux $\sh$ belonging to a discrete subspace of $\Hdiv$ that satisfies the equilibration identity $\Dv \sh + \kappa^2 \uh = \fh$, where $\fh$ is a piecewise polynomial approximation of $f$.
The estimator is then composed of terms of the form
\[
    \min\left\{\norm{\eps \nabla \uh + \eps^{-1} \frh}_\elm, C \eps^{\mft} \af^\ft \norm{j_\idx}_{\pt \elm\setminus\pt\Omega}\right\}.
\]
Thus it can be seen as a {\em combination} between an equilibrated flux estimator for diffusion problems and the residual-based estimator of~\cite{Verf_RD_rob_a_post_98} for reaction--diffusion problems. No submesh is needed for the construction of the estimator.
Subsequently, Ainsworth and Vejchodsk{\'y}~\cite{Ainsw_Vej_guar_rob_RD_11,Ainsw_Vej_guar_rob_RD_14} proceed in two stages.
First, equilibrated face fluxes are computed as in~\cite{Ains_Bab_rel_rob_a_post_RD_99}, and then, equilibrated fluxes are obtained by face liftings, so that the final estimate $\eta(\uh)$ is also fully computable and the first inequality in~\eqref{eq_est_rel_eff} is guaranteed with  $C_{\mathrm{rel}} = 1$. As in~\cite{Ains_Bab_rel_rob_a_post_RD_99}, though, boundary-layer adapted submeshes appear in the construction of the estimator.

The use of a submesh complicates the construction and implementation of the equilibrated flux estimators of~\cite{Ainsw_Vej_guar_rob_RD_11,Ainsw_Vej_guar_rob_RD_14}. Moreover, it is likely to be even more involved when moving beyond lowest-order approximations.
In this work, by further developing the idea in~\cite{Ched_Fuc_Priet_Voh_guar_rob_FE_RD_09}, we show how to obtain {\em simple}, i.e. avoiding any submesh, yet {\em robust} equilibrated flux estimators for {\em arbitrary-order} approximations.
The a posteriori error estimates presented in this paper are based on a locally computable flux~$\sh$ and potential approximation $\phih$, respectively belonging to discrete subspaces of $\Hdv$ and $L^2(\Om)$ of the current mesh $\Th$, that satisfy the key {\em equilibration property}
\begin{equation}\label{eq:equilibration}
  \Dv\sh + \kappa^2 \phih = \Pih f,
\end{equation}
where $\Pih \colon L^2(\Om) \tends \PP_p(\Th)$ denotes the $L^2$-orthogonal projection operator.
The upper bound on the error then has the simple form
\begin{equation}\label{eq:bound_simplified}
\NORM{u-\uh}^2  \leq\sum_{\elm\in\Th} \left[ \wk \norm{\eps \nabla  \uh + \eps^{-1} \sh}_\elm +  \norm{\kappa\left(\uh-\phih\right)}_\elm + \wkt \norm{f-\Pih f}_{\elm}\right]^2,
\end{equation}
where $\wk$ is an elementwise computable {\em weight} (cut-off factor) such that
\[
    \wk = \min\left\{1, C_* \sqrt{\frac{\eps}{ \kappa h_\elm} } \right\}, \quad \wkt = \min\left\{\frac{h_{\elm}}{\pi\eps}, \frac{1}{\kappa} \right\},
\]
with a fixed computable constant $C_*$ given by~\eqref{eq_C_st}; see Theorem~\ref{thm:upper_bound} below for further details.
The equilibrated flux $\sh$ and approximate potential $\phih$ in~\eqref{eq:equilibration}, \eqref{eq:bound_simplified} are obtained by an extension of the patchwise equilibration of~\cite{Dest_Met_expl_err_CFE_99,Braess_Scho_a_post_edge_08}, see also~\cite{Brae_Pill_Sch_p_rob_09,Ern_Voh_p_rob_15}.

Furthermore, we prove robustness and efficiency of the estimator~\eqref{eq:bound_simplified} by showing that its local contributions are bounded, up to a constant, by the local residual estimators. More precisely, for each $\elm\in\Th$, we show that
\begin{equation}\label{eq:intro_residual_lower_bound}
\begin{aligned}
  \wk^2 \norm{\eps \nabla  \uh + \eps^{-1} \sh}_\elm^2 +  \norm{\kappa\left(\uh-\phih\right)}_\elm^2
 & \lesssim \sum_{\elmt\in\frakT_\elm} \akp^2 \norm{r_\idx}_{\elm^{\prime}}^2 + \sum_{F\in\frakF_\elm} \eps^{-1} \af\norm{j_\idx}_{F}^2
\\ &  \lesssim \sum_{\elmt\in\frakT_\elm} \left[ \NORM{u-\uh}_{\elmt}^2 + \alpha_{\elmt}^2 \norm{f-\Pih f}_{\elmt}^2 \right],
 \end{aligned}
\end{equation}
where $\frakT_\elm$ and $\frakF_\elm$ denote the set of elements and faces in a suitable neighbourhood of $\elm$ and
\be \label{eq_en_el}
    \NORM{v}^2_\elm \eq \eps^2 \norm{\nabla v}^2_\elm + \kappa^2 \norm{v}^2_\elm \qquad \elm \in \Th,
\ee
see Proposition~\ref{prop:residual_lower_bound} and Theorem~\ref{thm:efficiency_robustness} below for full details. Crucially, the constants hidden in $\lesssim$ in~\eqref{eq:intro_residual_lower_bound} are independent of the mesh-sizes $h_\elm$ and problem parameters $\eps$ and $\kappa$, depending only on the shape-regularity of $\Th$, the space dimension $d$, and the polynomial degree $p$.
Hence, just as for residual-based estimates, equilibrated flux estimates have a {\em straightforward extension} from the pure diffusion case $\kappa = 0$, based on including appropriate weights (cut-off factors) and not requiring computations of quantities over any submesh or combination with the residual estimators.
In light of these results, we believe that the claims in~\cite{Verf_note_const_free_09,Verf_13} of a ``structural defect'' of the robustness of the equilibrated fluxes estimators are not generally valid.

As a side result, we also prove in Proposition~\ref{prop:flux_lower_bound} that the weights $\wk$ in~\eqref{eq:bound_simplified} are {\em necessary} for robustness of {\em any} equilibrated flux estimate involving the terms $\norm{\eps \nabla \uh + \eps^{-1} \frh}_\elm$ whenever $\sh$ is a piecewise polynomial on $\Th$ (and thus its construction does not involve any submesh), regardless of the precise details of the construction of $\sh$.
This proves that several flux equilibrations proposed in the past cannot be robust with respect to reaction dominance in general (although in many constellations, no loss of robustness may be numerically observed), including those of Repin and Sauter~\cite{Rep_Sau_funct_a_post_react_dif_06}, Ainsworth~\eal\ \cite{Ains_All_Bar_Rank_a_post_ADR_13}, Eigel and Samrowski~\cite{Eig_Sam_func_a_post_RD_14}, Eigel and Merdon~\cite{Eig_Mer_equil_ADR_16}, and~Vejchodsk\'y~\cite{Vej_compl_a_post_12,Vej_RD_ENUMATH_17,Vej_qual_equil_15}.

We only treat isotropic meshes. Results for anisotropic meshes can be found in Kunert~\cite{Kun_rob_a_post_sing_pert_RD_01},
Grosman~\cite{Gros_eqil_res_comp_err_RD_anis_06}, Apel {\em et
al.}~\cite{Ap_Nic_Sirch_anis_ADR_11}, Zhao and
Chen~\cite{Zhao_Chen_a_post_RD_anis_14}, or Kopteva~\cite{Kopt_a_post_RD_inf_anis_17,Kopt_a_post_equi_fl_anis_RD_17}.
Also, we are solely interested in the energy norm.
Robust estimates in the maximum norm are obtained in Demlow and Kopteva~\cite{Dem_Kopt_a_post_RD_inf_16} and, on possibly anisotropic meshes, in Kopteva~\cite{Kopt_a_post_RD_inf_15} for $p=1$ any in Linss~\cite{Lins_RD_HO_inf_14} for any order $p \geq 1$ in one space dimension. We refer to Stevenson~\cite{Stev_cvg_RD_05} for robust convergence, and we refer to Faustmann and Melenk~\cite{Faus_Mel_exp_cvg_RD_17} and the references therein for balanced norms.
Finally, extensions to variable coefficients $\eps$ and $\kappa$ can be treated easily as in~\cite{Ainsw_Vej_guar_rob_RD_14}, whereas inhomogeneous Dirichlet and Neumann boundary conditions, mixed parallelepipedal--simplicial meshes, meshes with hanging nodes, and approximations with varying polynomial degree $p$ can be treated as in Dolej{\v{s}}{\'{\i}} {\em et al.}~\cite{Dol_Ern_Voh_hp_16}.

\section{Construction of the equilibrated flux}\label{sec:construction}

We present in this section the construction of our equilibrated flux $\sh$ and of the potential approximation $\phih$.

\subsection{Notation}

Let $\Th$ be a matching simplicial partition of the domain $\Om$, \ie, $\bigcup_{\elm \in \Th} \elm = \overline \Om$, any element $\elm \in \Th$ is a closed simplex (interval when $d=1$, triangle when $d=2$, tetrahedron when $d=3$), and the intersection of two different simplices is either empty, or a vertex, or their common $l$-dimensional face, $1 \leq l \leq d-1$.
We denote by $\sigma_\T>0$ the shape-regularity parameter of the mesh $\Th$, i.e.\
\begin{equation}\label{eq:regularity_parameter}
\sreg \coloneqq \max_{K\in\T} \frac{h_K}{\rho_K},
\end{equation}
where $\rho_K$ is the diameter of the largest ball contained in $K$.
For each element $\elm\in \Th$ and for a fixed integer $p\geq 1$, let $\PP_p(\elm)$ denote the space of polynomials of total degree at most $p$ on $\elm$.
Let
\[
    \PP_p(\Th) \eq \{ v\in L^2(\Om),\; v|_\elm \in \PP_p(\elm) \quad \forall \elm\in\Th\}
\]
denote the space of scalar piecewise polynomials of degree at most $p$ over $\Th$.
Let $\Pih \colon L^2(\Om) \tends \PP_p(\Th)$ denote the $L^2$-orthogonal projection operator from $L^2(\Om)$ onto $\PP_p(\Th)$.
We additionally consider $\bL\eq L^2(\Om;\R^\dim)$ and  $\RTNp\subset \bL$ the piecewise Raviart--Thomas--N\'ed\'elec space defined by
\be \bs \label{eq_RTN_p}
    \RTNp & \eq \{\bvh \in \bL,\;\bvh|_{\elm} \in \RTNK \quad \forall \elm\in\Th\}, \\
    \RTNK & \eq \PP_p(\elm;\R^\dim) + \PP_p(\elm)\bm{x}.
\es \ee

For any subset $S$ of $\overline{\Om}$, let $h_S$ denote the diameter of $S$. Thus, for instance, $h_\elm$ denotes the diameter of the element $\elm\in\Th$.
Let $\calV$ denote the set of vertices of the mesh $\Th$. It is partitioned into the set of interior vertices $\Vint \eq \{ \ver\in\calV,\quad \ver\in\Om\} $, and boundary vertices $\Vext \eq \calV\setminus\Vint$.
For each vertex $\ver \in \calV$, the function $\psia$ is the hat function associated with $\ver$, \ie, $\psia \in \PP_1(\Th) \cap \Ho$ taking value $1$ in the vertex $\ver$ and $0$ in the other vertices. The set~$\oma$ is the interior of the support of $\psia$ with associated diameter $h_{\oma}$.
Furthermore, let $\Ta$ denote the restriction of the mesh $\Th$ to~$\oma$, and let $\calFa$ denote the set of interior faces of $\Ta$, i.e. the faces of $\Ta$ that contain the vertex $\ver$ for $\ver \in \Vhint$, without those on $\pt \Om$ for $\ver \in \Vhext$.
For each element $\elm\in \Th$, we collect in $\VK$ the set of vertices of $\calV$ belonging to $\elm$. We also define $\frakT_\elm \eq \bigcup_{\ver\in\calV_\elm} \Ta$ and $\frakF_\elm \eq \bigcup_{\ver\in\calV_\elm} \calFa$.

Throughout this work, the notation $a\lesssim b$ means that $a\leq C b$ with a constant $C$ that only depends on the shape-regularity parameter $\sreg$ of $\Th$, on the space dimension $d$, and on the polynomial degree $p$, so that it is in particular independent of the mesh-sizes $h_{\elm}$ and of the problem parameters $\eps$ and $\kappa$; $a\simeq b$ then stands for $a\lesssim b$ and simultaneously $b\lesssim a$.

\subsection{Trace and inverse inequalities} \label{sec_ineqs}

We first recall two inequalities that we will rely on.

\begin{lemma}[Trace inequality with explicit constant]
For all $\elm\in \Th$ and for all $v\in H^1(\elm)$ that satisfy $(v,1)_{\elm} =0$, i.e., that have vanishing mean-value on $\elm$, there holds
\begin{equation}\label{eq:trace_inequality}
\begin{aligned}
\norm{v}_{\partial\elm} \leq \Ctr \norm{\nabla v}_{\elm}^\ft\norm{v}_{\elm}^\ft, &&&  \Ctr \coloneqq \sqrt{\sreg (d+1) \left( 2 + d/\pi \right) }.
\end{aligned}
\end{equation}
\end{lemma}
\begin{proof}
We refer the reader to \cite[Lemma~1.49]{Di_Pietr_Ern_book_12} for the explicit constants of the trace inequality for general functions in $H^1(\elm)$; namely, for each face $F\subset \p \elm$,
\[
\norm{v}^2_{F} \leq \sreg \left( 2\norm{\nabla v}_K + d/h_K \norm{v}_{\elm}\right) \norm{v}_{\elm}.
\]
Then, we additionally apply the Poincar\'e inequality $\norm{v}_{\elm}\leq h_K/\pi \norm{\nabla v}_{\elm}$ for functions with vanishing mean-value on $\elm$, and sum over all the faces $F$ to obtain \eqref{eq:trace_inequality}.
\end{proof}

\begin{lemma}[Inverse inequalities with explicit constants]
For any $K\in \T$ and any $\bm{v}\in \RTNK$, we have
\begin{equation}\label{eq:inverse_inequality}
\begin{aligned}
h_K^{1/2}\norm{\bm{v}\cdot \bm{n}}_{\p K} &\leq \Cpdp \norm{\bm{v}}_K, &&&
 h_K \norm{\nabla \cdot \bm{v}}_K &\leq \CpdK \norm{\bm{v}}_K,
\end{aligned}
\end{equation}
where the constants $\Cpdp$ and $\CpdK$ are given by
\begin{align}
\Cpdp &\coloneqq \sqrt{(d+1)(p+2)(p+d+1)\sreg},  \\
\CpdK & \coloneqq \sqrt{d}\, \sreg \frac{\sqrt{5}}{4}\, (2\sqrt{2})^{\dim} \sqrt{p (p+1) (p+2) (p+3).}
\end{align}
\end{lemma}
\begin{proof}
See Appendix~\ref{sec:appendix}.
\end{proof}

In practice, possibly sharper constants can be obtained for the inequalities in \eqref{eq:inverse_inequality} by solving numerically small eigenvalue problems on each mesh element, or on a reference element in combination with bounds for the influence of the affine mapping.

We will need below the following constant composed of the constants of the
trace and inverse inequalities~\eqref{eq:trace_inequality} and~\eqref{eq:inverse_inequality}:
\be \label{eq_C_st}
  C_* \coloneqq \frac{1}{\sqrt{2}}\left( \frac{1}{\sqrt{\pi}}\CpdK + \Ctr\, \Cpdp  \right).
\ee

\subsection{Equilibrated flux $\sh$ and postprocessed potential $\phih$}

The construction of the auxiliary variables $\sh$ and $\phih$ giving the equilibration~\eqref{eq:equilibration} is based on independent local mixed finite element approximations of residual problems over the patches of elements around mesh vertices.

For each $\ver\in\calV$, let $\PP_p(\Ta)$, respectively $\RTNa$, be the restriction of the space~$\PP_p(\Th)$, respectively $\RTNp$, to the patch $\Ta$ around the vertex $\ver$. The local mixed finite element spaces $\Va$ and $\Qa$ are defined by
\bse \label{eq_spaces} \begin{align}
\Va & \eq
\begin{cases}
   \left\{\bvh \in \Hdiva \cap \RTNa,\, \bvh\scp \bn =0\text{ on }\p\oma \right\} & \, \text{if }\ver\in\Vint,\\
    \left\{\bvh \in \Hdiva\cap \RTNa ,\, \bvh\scp \bn =0\text{ on }\p\oma\setminus\p\Om  \right\}& \, \text{if }\ver\in\Vext,
\end{cases}
\label{eq_spaces_V} \\
\Qa & \eq \begin{cases}
   \PP_p(\Ta) & \, \text{if } \kappa > 0 \text{ or } \ver\in\Vext,\\
    \left\{ q_\idx\in \PP_p(\Ta),\, (q_\idx,1)_\oma = 0\right\}& \, \text{if } \kappa = 0  \text{ and }\ver\in\Vint,
\end{cases} \label{eq_spaces_Q}
\end{align} \ese
see Figure~\ref{fig:patches}.

\begin{figure}
\centerline{\includegraphics[width=0.4\textwidth]{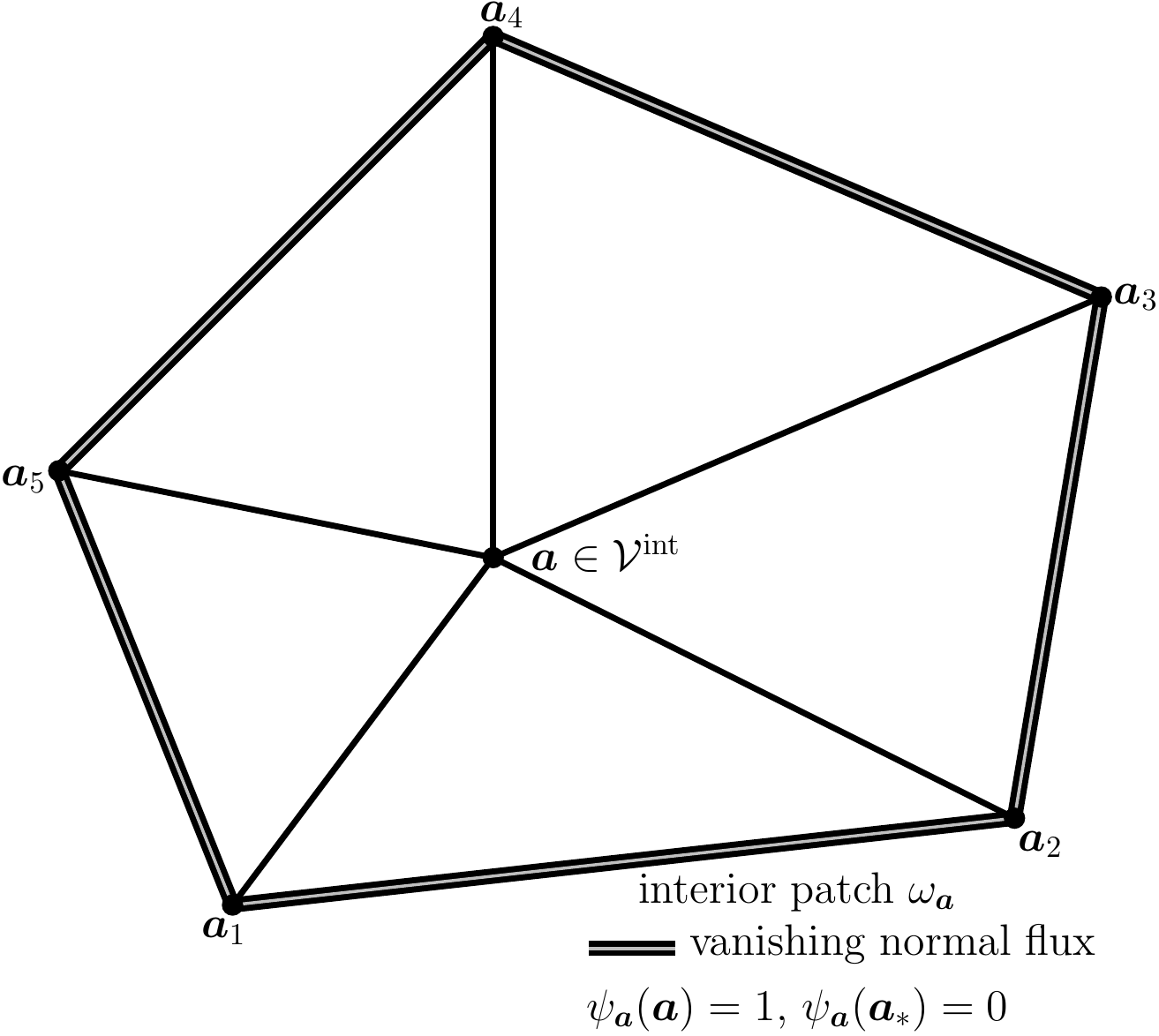} \quad \includegraphics[width=0.4\textwidth]{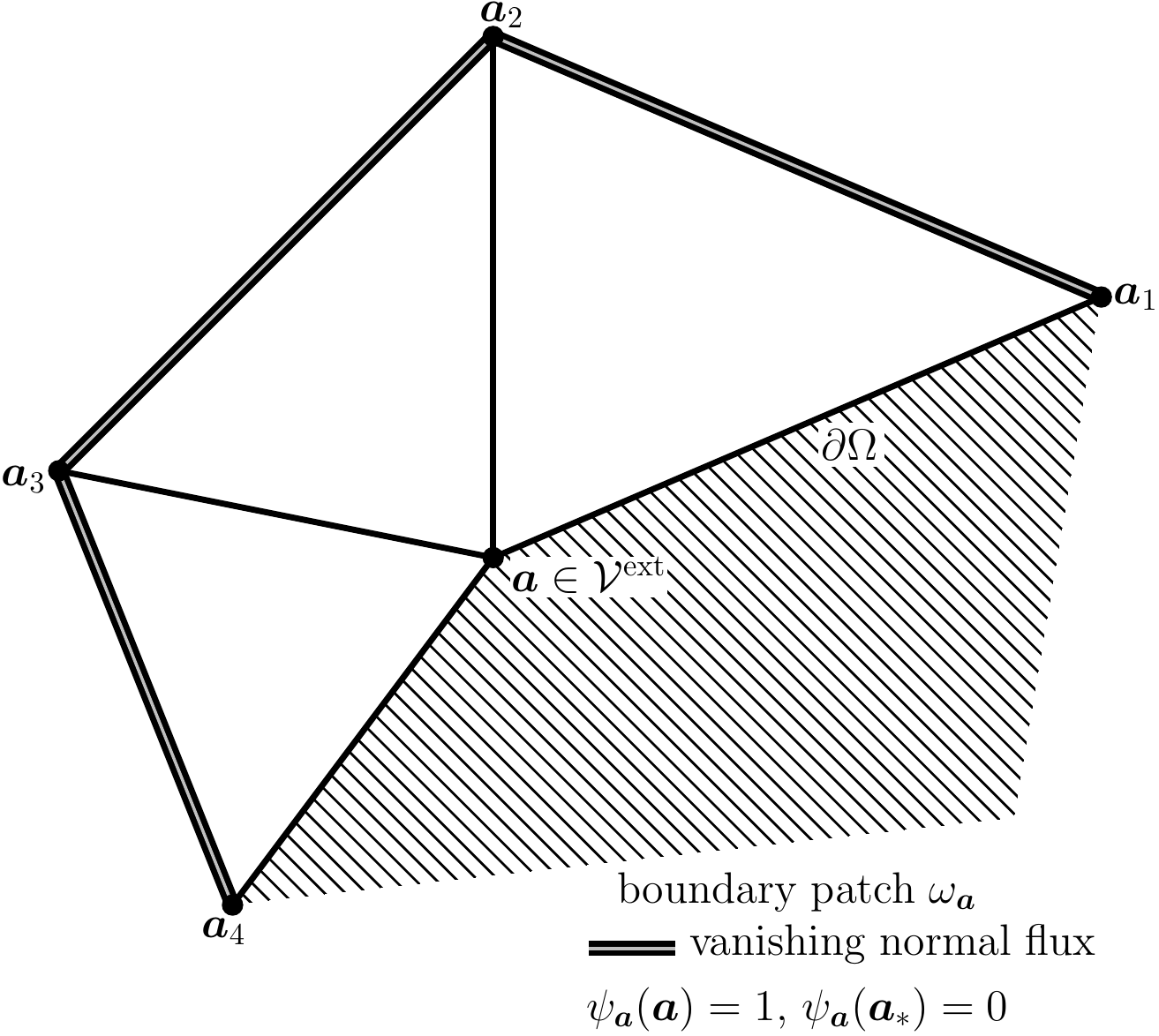}}
\caption{Patches $\Ta$, vanishing normal flux conditions in the local Raviart--Thomas--N\'ed\'elec spaces $\Va$, and hat functions $\psia$: interior (left) and boundary (right) vertex $\ver \in \calV$}
\label{fig:patches}
\end{figure}

Recall that $\uh \in \Vh$ with $\Vh = \PP_p(\Th)\cap \Hoo$ is the finite element solution given by~\eqref{eq:RD_fem}. Let $C_*$ be the constant composed of the constants of the trace and inverse inequalities and given by~\eqref{eq_C_st}. Our construction is:

\bd[Flux $\sh$ and potential $\phih$] \label{def:recs} For each vertex $\ver\in \calV$, let $(\sha,\phia) \in \Va\times \Qa$ be defined by the local constrained minimization problem
\bse \begin{equation}\label{eq:sha_def}
(\sha,\phia) \eq \argmin_{\substack{(\bvh,q_\idx) \in \Va\times \Qa \\ \Dv\bvh + \kappa^2 q_\idx  = \Pih(f \psia)- \eps^2 \nabla \uh \scp \nabla \psia}} \woma^2 \norm{\eps\psia \nabla \uh + \eps^{-1}\bvh}_{\oma}^2 +  \norm{ \kappa\left[\Pih(\psia \uh) - q_\idx\right]}^2_{\oma}
\end{equation}
with the weight
\be \label{eq:weight}
    \woma \eq  \min\left\{1, C_* \sqrt{\frac{\eps}{\kappa h_{\oma}}}\right\}.
\ee
Then, extending each $\sha$ and $\phia$ by zero outside of the patch $\oma$, $\sh\in \RTNp$ and $\phih\in \PP_p(\Th)$ are given by
\begin{equation}\label{eq:sh_def}
\sh \eq \sum_{\ver\in\calV} \sha,   \quad \phih \eq \sum_{\ver\in\calV} \phia.
\end{equation} \ese
\ed

We remark that for an interior vertex $\ver\in\Vint$, we have
\be \label{eq_Neu_comp}
    (\Pih(f \psia)- \eps^2 \nabla \uh \scp \nabla \psia,1)_{\oma} =  (f,\psia)_\oma - \eps^2 (\nabla \uh, \nabla \psia)_{\oma} =  \kappa^2 (\uh,\psia)_{\oma}
\ee
by Galerkin orthogonality with $\psia \in \Vh$ as a test function in~\eqref{eq:RD_fem}. Since
\[
    (\Dv\sha,1)_\oma = (\sha \scp \tn_\oma,1)_{\pt \oma} = 0
\]
by Green's theorem and the vanishing normal flux condition imposed in the definition~\eqref{eq_spaces_V} of $\Va$, it follows that $\phia$ necessarily satisfies the mean-value property
\[
    (\phia,1)_{\oma}=(\psia \uh,1)_{\oma} \qquad \forall \ver\in\Vint,
\]
 whenever $\kappa>0$.
If $\kappa = 0$ instead, then $\phia$ is undefined by~\eqref{eq:sha_def} but one remarks that it is no longer needed anywhere in the paper. In this case, Definition~\ref{def:recs} coincides with~\cite[equation~(9)]{Brae_Pill_Sch_p_rob_09}, \cite[Definition~6.9]{Ern_Voh_adpt_IN_13}, or~\cite[Construction~3.4]{Ern_Voh_p_rob_15}; in particular, the Neumann compatibility condition of problem~\eqref{eq:sha_def} for $\ver\in\Vint$ follows from~\eqref{eq_Neu_comp}.

In practice, the constrained minimization problem~\eqref{eq:sha_def} is solved through its Euler--Lagrange equations, which can be reduced to solving a linear system of dimension $\Dim \Va + \Dim \Qa$ in the present context. This problem reads: find $(\sha,\phia) \in \Va\times \Qa$ with $\phia = \ghia + \Pih(\psia \uh)$ and $(\sha,\ghia) \in \Va \times \Qa$ such that
\bse \label{eq_fl_equil} \bat{2}
    \eps^{-2} \woma^2 (\sha, \bv_\idx)_\oma - (\ghia, \Dv \bv_\idx)_\oma & = - \woma^2 (\psia  \Gr \uh, \bv_\idx)_\oma
    & \qquad & \forall \bv_\idx \in \Va, \label{eq_fl_equil_1}\\
    (\Dv \sha,q_\idx)_\oma + \kappa^2 (\ghia, q_\idx)_\oma & = (f \psia - \kappa^2 \psia \uh - \eps^2 \Gr \uh \scp \Gr \psia, q_\idx)_\oma &
    \qquad &\forall q_\idx \in \Qa. \label{eq_fl_equil_2}
\eat\ese

\subsection{Properties of $\sh$ and $\phih$}

We have constructed $\sh$ and $\phih$ such that the following holds:

\begin{proposition}[$\Hdv$-conformity of $\sh$, equilibration] \label{prop:equilibration}
Let $\sh \in \RTNp$ and $\phih\in \PP_p(\Th)$ be given by Definition~\ref{def:recs}. Then $\sh$ belongs to $\Hdv$, and $\sh$ and $\phih$ satisfy the equilibration property~\eqref{eq:equilibration}.
\end{proposition}

\begin{proof}
First, the $\Hdv$-conformity of $\sh$ follows from the fact that, for any vertex $\ver\in\calV$, the zero extension of $\sha$ belongs to $\Hdv$ as a result of the vanishing normal flux boundary conditions in the space $\Va$.
Then, to show~\eqref{eq:equilibration}, we employ the constraint in~\eqref{eq:sha_def} together with~\eqref{eq:sh_def}:
\[
    \Dv \sh + \kappa^2 \phih = \sum_{\ver\in\calV} [ \Dv \sha+\kappa^2\phia] = \sum_{\ver\in\calV} \big[\Pih (f \psia) - \eps^2 \nabla \uh \scp \nabla \psia \big] = \Pih f,
\]
where we have used the fact that the hat functions $\{\psia\}_{\ver\in \calV} $ form a partition of unity over $\Om$, i.e.\ $\sum_{\ver\in\calV} \psia = 1$.
\end{proof}

\section{A computable guaranteed a posteriori error estimate} \label{sec:main_estimate}

This section presents our guaranteed and fully computable a posteriori error estimate. The following upper bound on the energy norm of the error builds on~\cite[Theorems~3.1 and 4.4]{Ched_Fuc_Priet_Voh_guar_rob_FE_RD_09} and \cite[Lemma~2]{Ainsw_Vej_guar_rob_RD_14}. It employs additionally  the concept of a potential reconstruction $\phih$ that will turn out crucial for a simple and robust flux equilibration. Moreover, it relies on the trace and inverse inequalities of Section~\ref{sec_ineqs} to make appear the crucial weighs (cut-off factors), with the constant $C_*$ given by~\eqref{eq_C_st}.

\begin{theorem}[Guaranteed a posteriori error estimate]
\label{thm:upper_bound}
Let $u$ be the weak solution of problem~\eqref{eq_RD} given by~\eqref{eq:RD_weak} and let $\uh \in \Vh $ be its finite element approximation given by~\eqref{eq:RD_fem}.
Let $\sh \in \RTNp \cap \Hdv$ and $\phih \in \PP_p(\Th)$ be given by Definition~\ref{def:recs}.
Then the following upper bound for the energy norm of the error holds:
\begin{equation} \label{eq:upper_bound}
\NORM{u-\uh}^2 \leq \sum_{\elm \in \Th} \left[ \wk \norm{\eps \nabla  \uh + \eps^{-1} \sh}_\elm + \norm{\kappa\left(\uh-\phih\right)}_\elm  + \wkt \norm{f-\Pih f}_\elm \right]^2,
\end{equation}
where the weights $\wk$ and $\wkt$ are respectively defined by
\begin{equation} \label{eq_weights}
\wk \eq \min\left\{1, C_* \sqrt{\frac{\eps}{\kappa h_{\elm}}} \right\}, \quad
\wkt \eq \min\left\{\frac{h_{\elm}}{\pi\eps}, \frac{1}{\kappa} \right\}, \qquad \elm\in\Th.
\end{equation}
\end{theorem}
\begin{proof}
First, we note that the energy norm of the error $\NORM{u-\uh}$ is related to the residual $\calR(\uh)\in H^{-1}(\Om)$, defined by
\[
    \pair{\calR(\uh)}{v} \eq (f,v)-a(\uh,v), \qquad v\in \Hoo,
\]
through the identity
\begin{equation}\label{eq:dual_norm_residual}
\NORM{u-\uh} = \NORM{R(\uh)}_{*}, \quad \NORM{R(\uh)}_{*} \eq \sup_{v\in \Hoo,\,\NORM{v}=1} \pair{\calR(\uh)}{v},
\end{equation}
\cf, \eg, \cite[equation~(4.1)]{Verf_RD_rob_a_post_98}.
Consider now $\pair{\calR(\uh)}{v}$ for a fixed function $v \in \Hoo$.
Since $\sh \in \Hdv$ and $v \in \Hoo$, Green's theorem gives $(\sh,\nabla v) + (\Dv \sh,v) = 0$, so
\begin{equation}\label{eq:residual_equilibrated}
\pair{\calR(\uh)}{v} = (f,v)-a(\uh,v) = (f - \Pih f, v) + (\kappa (\phih - \uh), \kappa v) - (\eps \nabla \uh + \eps^{-1} \sh, \eps\nabla v),
\end{equation}
where we have also used the equilibration identity~\eqref{eq:equilibration}.
We now proceed by estimating each term in~\eqref{eq:residual_equilibrated} elementwise.

For each element $\elm\in \Th$, we use the identity $(f - \Pih f, v)_\elm = (f - \Pih f, v-\Pih v)_\elm$ and the Poincar\'e--Friedrichs inequality on the convex element $\elm$, \ie\ $\norm{v-\Pih v}_\elm\leq \frac{h_\elm}{\pi}\norm{\nabla v}_\elm$ for any $v\in H^1(\elm)$, together with the energy error definition~\eqref{eq_en_el}, to obtain the following bound
\begin{equation}\label{eq:data_osc_bound}
\abs{(f - \Pih f, v)_\elm } \leq \norm{f-\Pi_\idx f}_\elm \min\left\{\frac{h_\elm}{\pi \eps} \norm{\eps\nabla v}_\elm, \frac{1}{\kappa} \norm{\kappa v}_\elm \right\}\leq \wkt  \norm{f-\Pi_\idx f}_\elm \NORM{v}_\elm.
\end{equation}
Here, actually, a little sharper bound is possible by a convex combination of the two possibilities, but we prefer to use the simple form~\eqref{eq:data_osc_bound} with $\wkt$ in the form of minimum given by~\eqref{eq_weights}.

Next, it is clear that
\be \label{eq_0}
    \abs{ (\eps \nabla \uh + \eps^{-1} \sh, \eps\nabla v)_{\elm} } \leq \norm{ \eps \nabla \uh + \eps^{-1}\sh}_\elm \NORM{v}_\elm
\ee
for each $\elm\in \Th$. However, this is not necessarily the sharpest possible estimate in the singularly perturbed regime $\kappa \gg \eps$. Therefore, following the idea of~\cite[Proof of Theorem~4.4]{Ched_Fuc_Priet_Voh_guar_rob_FE_RD_09}, we use Green's theorem elementwise together with the fact that $\nabla v = \nabla(v-\overline{v}_{\elm})$, where $\overline{v}_{\elm}$ denotes the mean-value of $v$ on $\elm$. This gives
\[
(\eps \nabla \uh + \eps^{-1} \sh, \eps\nabla v)_{\elm} =  ( (\eps \nabla \uh + \eps^{-1} \sh)\scp \bn , \eps (v-\overline{v}_{\elm}))_{\p \elm} - ( \Dv\left(\eps \nabla \uh + \eps^{-1} \sh\right), \eps (v-\overline{v}_{\elm}))_{\elm}.
\]
The $\Lti{\elm}$-stability of the mean-value,  $\norm{v-\overline{v}_{\elm}}_\elm \leq \norm{v}_\elm$, Young's inequality
\be \label{eq_You_en}
    \norm{\eps\nabla v}_{\elm}^\ft\norm{\kappa v}^\ft_\elm \leq \frac{1}{\sqrt{2}}\NORM{v}_{\elm},
\ee
and the multiplicative trace inequality~\eqref{eq:trace_inequality} altogether lead to
\[
    \eps \norm{v-\overline{v}_{\elm}}_{\p \elm} \leq \Ctr \eps^{\ft} \norm{\eps \nabla v}_{\elm}^\ft\norm{v}_{\elm}^\ft \leq \frac{ \Ctr}{\sqrt{2}} \sqrt{h_K}\sqrt{\frac{\eps}{\kappa h_{\elm}}} \NORM{v}_{\elm}.
\]
Combined with the inverse inequality~\eqref{eq:inverse_inequality}, we find that
\be \label{eq_1}
    \abs{( (\eps \nabla \uh + \eps^{-1} \sh)\scp \bn , \eps (v-\overline{v}_{\elm}))_{\p \elm}}
     \leq \frac{\Cpdp \Ctr}{\sqrt{2}}  \sqrt{\frac{\eps}{\kappa h_{\elm}}} \norm{\eps \nabla \uh + \eps^{-1} \sh}_{\elm} \NORM{v}_\elm.
\ee
The $\Lti{\elm}$-stability of the mean-value, the Poincar\'e--Friedrichs inequality in the form $\norm{v-\overline{v}_{\elm}}_\elm \leq \frac{h_\elm}{\pi}\norm{\nabla v}_\elm$, and~\eqref{eq_You_en} yield
\[
    \eps \norm{v-\overline{v}_{\elm}}_{\elm} \leq \eps \norm{v}_{\elm}^\ft
    h_\elm^\ft \pi^\mft \norm{\nabla v}_\elm^\ft \leq   \frac{h_\elm}{\sqrt{2\pi} } \sqrt{\frac{\eps}{\kappa h_{\elm}}} \NORM{v}_{\elm}.
\]
Thus, combined with the inverse inequality~\eqref{eq:inverse_inequality}, we find that
\be \label{eq_2}
\abs{( \Dv\left(\eps \nabla \uh + \eps^{-1} \sh\right), \eps (v-\overline{v}_{\elm}))_{\elm}} \leq \frac{1}{\sqrt{2}} \frac1{\sqrt{\pi}}\CpdK  \sqrt{\frac{\eps}{\kappa h_{\elm}}} \norm{\eps \nabla \uh + \eps^{-1} \sh}_{\elm} \NORM{v}_{\elm}.
\ee
Therefore, combining inequalities~\eqref{eq_0}, \eqref{eq_1}, and~\eqref{eq_2}, we get
\begin{equation}\label{eq:flux_bound_1}
\abs{(\eps \nabla \uh + \eps^{-1} \sh, \eps\nabla v)_{\elm}} \leq \wk \norm{\eps \nabla \uh + \eps^{-1} \sh}_{\elm} \NORM{v}_{\elm} \quad \forall \elm\in\Th
\end{equation}
with $\wk$ given by~\eqref{eq_weights} and $C_*$ given in \eqref{eq_C_st}. As a side remark, it is possible to obtain a slightly sharper, at the expense of making the weight $\wk$ more complicated than the simple form given by~\eqref{eq_weights}.

Finally, we can apply the Cauchy--Schwarz inequality to see that $|(\kappa(\phih - \uh),\kappa v)_{\elm}| \leq \norm{\kappa\left(\uh - \phih\right)}_{\elm}\NORM{v}_{\elm}$. Therefore, we deduce from~\eqref{eq:residual_equilibrated} and the above inequalities that
\[
\abs{\pair{\calR(\uh)}{v}}\leq \sum_{\elm\in\Th} \left[ \wk \norm{\eps \nabla \uh + \eps^{-1} \sh}_{\elm} + \norm{\kappa\left(\uh - \phih\right)}_{\elm} + \wkt \norm{f-\Pih f}_{\elm}\right] \NORM{v}_{\elm},
\]
which implies the upper bound on the error~\eqref{eq:upper_bound} after another Cauchy--Schwarz inequality, using~\eqref{eq:dual_norm_residual} and $\sum_{\elm \in \Th} \NORM{v}_\elm^2 = \NORM{v}^2$.
\end{proof}

\section{Efficiency and robustness of the estimate}

This section establishes the local (and consequently global) efficiency and robustness of our a posteriori error estimate.

\subsection{A basic stability result}

The main tool in the analysis of efficiency is the following stability result, where, we recall, the broken and the patchwise ${\bm{H}}(\dv)$-conforming Raviart--Thomas--N\'ed\'elec spaces $\RTN_p(\Ta)$ and $\Va$ are respectively given by~\eqref{eq_RTN_p} and~\eqref{eq_spaces}.

\begin{lemma}[Stability of patchwise flux equilibration] \label{lem:main_stability_bound}
Let a vertex $\ver\in\calV$ be fixed, and let $g_{\cyee \idx}\in \PP_p(\Ta)$ and $\btauh \in \RTN_p(\Ta)$ be given discontinuous piecewise polynomial functions, with the Neumann compatibility condition $(g_{\cyee \idx},1)_{\oma} =0$ satisfied if $\ver\in\Vint$. Then, there holds
\begin{equation}\label{eq:main_stability_bound}
\min_{\substack{\bvh \in \Va\\ \Dv\bvh = g_{\cyee \idx}}} \norm{\btauh + \bvh}_{\oma}
\lesssim  \sup_{\substack{v\in H^1_*(\oma)\\ \norm{\nabla v}_{\oma}=1}}
 \big\{(g_{\cyee \idx},v)_{\oma} - (\btauh,\nabla v)_{\oma} \big\},
\end{equation}
where $H^1_*(\oma)$ is the subspace of functions in $H^1(\oma)$ that have mean-value zero on the patch subdomain $\oma$ if $\ver\in\Vint$ is an interior vertex, or that vanish on $\partial\oma\cap\partial\Om$ if $\ver\in\Vext$ is a boundary vertex.
\end{lemma}%

The above result holds for any dimension $\dim\geq 1$, although some additional properties are known for $\dim\leq 3$. Indeed, in the case where $\dim=2$, it is shown in~\cite[Theorem~7]{Brae_Pill_Sch_p_rob_09} that the constant in~\eqref{eq:main_stability_bound} is in fact independent of the polynomial degree $p$, \ie\ $p$-robust. The extension of the $p$-robustness of the bound to the case of $\dim=3$ was shown in \cite[Corollaries~3.3 and~3.6]{Ern_Voh_p_rob_3D_18}. It is also possible to extend similar results of this kind to situations with hanging nodes and locally refined submeshes, as shown in \cite{Ern_Smears_Voh_H-1_lift_17}.

\subsection{Stability with respect to residual estimators}

The next lemma shows that the local contributions of the equilibrated flux a posteriori estimators of Definition~\ref{def:recs} lie below the local residual estimators as defined in~\eqref{eq:residual_estimator}, with the element residuals $r_\idx$ and face residuals $j_\idx$ are defined by~\eqref{eq:res} and the weights $\ak$ and $\af$ defined by~\eqref{eq:residual_weights}.

\begin{lemma}[Stability of patchwise flux equilibration with respect to residual estimators]\label{lem:patch_residual_bound}
For each $\ver\in\calV$, let $\sha$ and $\phia$ be defined by~\eqref{eq:sha_def}. Then
\begin{equation}\label{eq:patch_residual_bound}
 \woma^2 \norm{\eps\psia \nabla \uh + \eps^{-1}\sha}_{\oma}^2 +  \norm{ \kappa\left[\Pih(\psia \uh) - \phia\right]}_{\oma}^2
  \lesssim \sum_{\elm\in \Ta} \ak^2 \norm{r_\idx}_{\elm}^2 + \sum_{F\in \calFa} \eps^{-1} \af \norm{j_\idx}_{F}^2.
\end{equation}
\end{lemma}

\begin{proof}
Let a vertex $\ver\in\calV$ be fixed. Since $\sha$ and $\phia$ are defined as minimizers of the functional in the right-hand side of~\eqref{eq:sha_def}, it is enough to prove that there always exists some $\bvh^*\in \Va$ and $q_\idx^*\in\Qa$ that satisfy the constraint $\Dv\bvh^* +\kappa^2 q_\idx^* = \Pih(f \psia)- \eps^2 \nabla \uh \scp \nabla \psia$ and that satisfy the bound~\eqref{eq:patch_residual_bound} with $\bvh^*$ in place of $\sha$ and $q_\idx^*$ in place of $\phia$. The specific construction depends on the mesh size and the problem parameters $\eps$ and $\kappa$, as we now show.

\emph{Case 1, $ \eps/h_{\oma} \leq \kappa$ (reaction dominance).} Up to a constant, we have $\kappa^{-1} \lesssim h_{\elm}/\eps$ and $\kappa^{-1} \lesssim h_{F}/\eps$ for all elements $\elm\in\Ta$ and all interior faces $F\in\calFa$.
In this case, we adopt the following construction.
Let
\[
    \rho_{\ver} \eq \frac{1}{\abs{\oma}}(\psia r_\idx,1)_{\oma} = \frac{1}{\abs{\oma}}(\psia (f+\eps^2 \Delta_\idx \uh - \kappa^2 \uh),1)_\oma, \qquad \ver\in\Vint,
\]
and $\rho_{\ver} \eq 0$ otherwise.
Next, we define
\begin{equation}
q_\idx^* \eq \frac{1}{\kappa^2} \left( \Pih \left(f \psia\right) + \eps^2 \psia \Delta_\idx \uh - \rho_{\ver} \right), \quad
\bvh^* \eq \argmin_{\substack{\bvh\in\Va \\ \Dv\bvh = g_{\cyee \idx}^* }} \norm{\eps^2 \psia \nabla \uh + \bvh}_{\oma},
\end{equation}
where
\[
    g_{\cyee \idx}^*\eq - \eps^2 (\nabla\uh\scp\nabla \psia + \psia \Delta_\idx \uh) + \rho_{\ver}.
\]
It is easy to check that if $\ver\in\Vint$, then $(g_{\cyee \idx}^{*},1)_{\oma} =0$, since the Galerkin orthogonality (take $\vh = \psia$ in~\eqref{eq:RD_fem}) implies that
\be\label{eq_GO_eff}
    (g_{\cyee \idx}^*,1)_{\oma}=(f,\psia)-\eps^2 (\nabla \uh,\nabla \psia)-\kappa^2(\uh,\psia) =0.
\ee
Therefore, it follows that $q_\idx^*\in \Qa $ and $\bvh^*\in \Va$ are well-defined and that they satisfy the constraint $\Dv \bvh^* + \kappa^2 q_\idx^* = \Pih(f \psia)- \eps^2 \nabla \uh \scp \nabla \psia$.

We now bound $\woma^2 \norm{\eps^2 \psia \nabla \uh + \bvh^*}_{\oma}^2$ and $\norm{\kappa\left[\Pih(\psia \uh)-q_\idx^*\right]}_{\oma}^2$.
First, we obtain
\[
\norm{\kappa\left[\Pih(\psia \uh)-q_\idx^*\right]}_{\oma}^2 = \frac{1}{\kappa^2}\norm{\Pih(\psia r_\idx)-\rho_{\ver}}_{\oma}^2 \leq \frac{1}{\kappa^2} \norm{r_\idx}_{\oma}^2 \lesssim \sum_{\elm\in\Ta} \ak^2 \norm{r_\idx}_{\elm}^2,
\]
where we have used the stability of the $L^2$-projection (note that $\rho_{\ver}$ is also the mean value of $\Pih(\psia r_\idx)$ on $\oma$ for $\ver \in \Vhint$) and the fact that $\norm{\psia}_{\infty, \oma}=1$ to bound $\norm{\Pih(\psia r_\idx)-\rho_{\ver}}_{\oma}$.
Next, we apply Lemma~\ref{lem:main_stability_bound} to bound $\woma^2 \norm{\eps^2 \psia \nabla \uh + \bvh^*}_{\oma}^2$. Note first that for an interior vertex $\ver \in \Vhint$, $(\rho_{\ver},v)_{\oma}=0$ since $v\in H^1_*(\oma)$ implies that $v$ is orthogonal to constant functions on $\oma$. We find that
\begin{equation} \label{eq_IPP_eff}
\begin{split}
\norm{\eps^2 \psia \nabla \uh + \bvh^*}_{\oma}
 &\lesssim \sup_{v\in H^1_*(\oma),\, \norm{\nabla v}_{\oma}=1} \big\{(g_{\cyee \idx}^*,v)_{\oma}-(\eps^2\nabla \uh,\psia\nabla v)_{\oma} \big\}
\\&=  \sup_{v\in H^1_*(\oma),\, \norm{\nabla v}_{\oma}=1} \big\{ -(\eps^2 \nabla \uh, \nabla(\psia v))_{\oma} - (\eps^2 \Delta_{\idx} \uh,\psia v)_{\oma} \big\}
\\&= \sup_{v\in H^1_*(\oma),\,  \norm{\nabla v}_{\oma}=1}  \sum_{F\in\calFa} (j_\idx,\psia v)_F,
\end{split}
\end{equation}
where the last line follows by elementwise integration by parts. It is then straightforward to deduce from the trace inequalities $\norm{v}_{F} \lesssim h_\elm^\mft \norm{v}_{\elm} + \norm{\nabla v}_{\elm}^\ft\norm{v}_{\elm}^\ft$ and the Poincar\'e--Friedrichs inequality for functions in $H^1_*(\oma)$ $\norm{v}_\oma \lesssim h_\oma \norm{\Gr v}_\oma$ that
\be \label{eq_jumps_eff}
    \norm{\eps^2 \psia \nabla \uh + \bvh^*}_{\oma}^2 \lesssim h_\oma \sum_{F\in\calFa} \norm{j_\idx}_F^2.
\ee
Consequently, using definition~\eqref{eq:weight} of the weight $\woma$
\[
\woma^2 \norm{\eps \psia \nabla \uh + \eps^{-1}\bvh^*}_{\oma}^2 \lesssim \frac{\eps}{\kappa h_{\oma}} \frac{h_{\oma}}{\eps^2} \sum_{F\in\calFa} \norm{j_\idx}_F^2 \lesssim  \sum_{F\in\calFa} \eps^{-1} \af \norm{j_\idx}_F^2.
\]
Therefore, if $\eps/h_{\oma} \leq  \kappa$, we have shown that there exists $\bvh^*$ and $q_\idx^*$ satisfying the constraint $\Dv \bvh^* + \kappa^2 q_\idx^* = \Pih(f \psia)- \eps^2 \nabla \uh \scp \nabla \psia$ and such that
\[
\woma^2 \norm{\eps \psia \nabla \uh + \eps^{-1}\bvh^*}_{\oma}^2 + \norm{\kappa\left[\Pih(\psia \uh)-q_\idx^*\right]}_{\oma}^2 \lesssim \sum_{\elm\in \Ta} \ak^2 \norm{r_\idx}_{\elm}^2 + \sum_{F\in \calFa} \eps^{-1}\af \norm{j_\idx}_{F}^2.
\]
As explained above, this implies~\eqref{eq:patch_residual_bound} in the case $\eps/h_{\oma} \leq  \kappa$.

\emph{Case 2, $\eps/h_{\oma} > \kappa$ (diffusion dominance).} We select
\[
    q_\idx^* \eq \Pih(\psia \uh), \quad
\bvh^* \eq \argmin_{\substack{\bvh\in\Va \\ \Dv\bvh = g_{\cyee \idx}^* }} \norm{\eps^2 \psia \nabla \uh + \bvh}_{\oma},
\]
where
\[
    g_{\cyee \idx}^* \eq \Pih(\psia(f-\kappa^2 \uh))-\eps^2 \nabla \psia \scp\nabla \uh.
\]
Notice that Galerkin orthogonality implies that $(g_{\cyee \idx}^*,1)_{\oma}=0$ if $\ver\in\Vint$ as in~\eqref{eq_GO_eff}, and also $\Dv\bvh^*+\kappa^2 q_\idx^* =\Pih(f \psia)- \eps^2 \nabla \uh \scp \nabla \psia$, so the requested constraint is satisfied.
It then follows directly from Lemma~\ref{lem:main_stability_bound} that
\[
 \norm{\eps^2 \psia \nabla \uh + \bvh^*}_{\oma}
 \lesssim \sup_{v\in H^1_*(\oma),\,\norm{\nabla v}_{\oma}=1} \left\{ (\Pih (\psia r_\idx),v)_{\oma} + \sum_{F\in\calFa} (\psia j_\idx,v)_F\right\},
\]
where we use the fact that elementwise integration by parts shows that, as in~\eqref{eq_IPP_eff},
\[
    (g_{\cyee \idx}^*,v)_{\oma}-(\eps^2\psia \nabla \uh,\nabla v)_{\oma} = (\Pih (\psia r_\idx),v)_{\oma} + \sum_{F\in\calFa} (\psia j_\idx,v)_F.
\]
Thus, proceeding as in~\eqref{eq_IPP_eff}--\eqref{eq_jumps_eff} for the face residuals term and using the stability of the $L^2$-projection, $\norm{\psia}_{\infty, \oma}=1$, and the Poincar\'e--Friedrichs inequality for functions in $H^1_*(\oma)$, $\norm{v}_\oma \lesssim h_\oma \norm{\Gr v}_\oma$, for the element residuals term, we get
\[
    \norm{\eps^2 \psia \nabla \uh + \bvh^*}_{\oma}^2 \lesssim h_\oma^2 \sum_{\elm\in \Ta} \norm{r_\idx}_{\elm}^2 + h_\oma \sum_{F\in\calFa} \norm{j_\idx}_F^2.
\]
Consequently,
\[
\norm{\eps \psia \nabla \uh + \eps^{-1}\bvh^*}_{\oma}^2 \lesssim  \sum_{\elm\in \Ta} \ak^2 \norm{r_\idx}_{\elm}^2 + \sum_{F\in\calFa} \eps^{-1}\af \norm{j_\idx}_F^2.
\]
Hence, on noting that $\woma\leq 1$ and that $\norm{\kappa\left[\Pih(\psia \uh)-q_\idx^*\right])}_{\oma}=0$, we see that~\eqref{eq:patch_residual_bound} also holds for the case $ \eps/h_{\oma} > \kappa$.
\end{proof}

Recall that $\frakT_\elm \eq \bigcup_{\ver\in\calV_\elm} \Ta$ and $\frakF_\elm \eq \bigcup_{\ver\in\calV_\elm} \calFa$.

\begin{proposition}[Bound on flux estimators by the residual estimators]\label{prop:residual_lower_bound}
Let $\sh$ and $\phih$ be given by Definition~\ref{def:recs}.
Additionally, let the volume and face residual functions $r_\idx$ and $j_\idx$ be defined by~\eqref{eq:res}.
Then, for each element $\elm\in\Th$, we have the bound
\begin{equation}\label{eq:residual_lower_bound}
\wk^2 \norm{\eps \nabla  \uh + \eps^{-1} \sh}_\elm^2 + \norm{\kappa\left(\uh-\phih\right)}_\elm^2 \lesssim \sum_{\elmt\in \frakT_\elm } \alpha_{\elmt}^2 \norm{r_\idx}_{\elmt}^2 + \sum_{F\in \frakF_\elm} \eps^{-1} \af \norm{j_\idx}_{F}^2.
\end{equation}
\end{proposition}

\begin{proof}
For each mesh element $\elm\in\Th$, we have $\sh|_{\elm} = \sum_{\ver\in\VK} \sha|_{\elm}$ and $\phih|_{\elm}=\sum_{\ver\in\VK}\phia|_{\elm}$. Furthermore, since $\{\psia\}_{\ver\in\VK}$ form a partition of unity over $\elm$ and since $\Pih$ is the elementwise $L^2$ projection of degree $p$, it follows that $\uh|_{\elm} = \Pih \uh|_{\elm} = \sum_{\ver\in\VK} \Pih\left( \psia \uh \right)|_{\elm}$.
Furthermore, \eqref{eq_weights} and~\eqref{eq:weight} together with the mesh shape regularity imply that $\wk \lesssim \woma$ for each $\ver\in\VK$, where the constant depends only on $\sreg$.
Therefore, we obtain
\ban
{} & \wk^2 \norm{\eps \nabla \uh + \eps^{-1}\sh}_{\elm}^2 + \norm{\kappa\left(\uh - \phih\right)}_{\elm}^2 \\
\lesssim {} & \sum_{\ver\in\VK}\left[ \woma^2 \norm{\eps \psia \nabla \uh + \eps^{-1}\sha}_{\elm}^2 + \norm{\kappa[\Pih(\psia \uh)-\phia]}_{\elm}^2 \right].
\ean
Therefore, we can use~\eqref{eq:patch_residual_bound} for each $\ver\in\VK$ to get~\eqref{eq:residual_lower_bound}.
\end{proof}

\subsection{Local efficiency and robustness of the estimate}

We now recall the well-known efficiency and robustness results for residual estimators, see~\cite[Proposition~4.1]{Verf_RD_rob_a_post_98} and \cite{Verf_13} for details. For each $\elm\in\Th$ and $F\in\calFint$, there holds
\begin{subequations}\label{eq:residual_efficiency}
\begin{align}
   \ak^2 \norm{r_\idx}_{\elm}^2  &\lesssim \NORM{u-\uh}_{\elm}^2 + \ak^2 \norm{f-\Pih f}_{\elm}^2,\\
   \eps^{-1}\af \norm{j_\idx}_F^2 & \lesssim \sum_{\elm\in\Th, F\subset \p \elm} \left[\NORM{u-\uh}_{\elm}^2 + \ak^2 \norm{f-\Pih f}_{\elm}^2 \right].
\end{align}
\end{subequations}

Therefore, the combination of Proposition~\ref{prop:residual_lower_bound} with~\eqref{eq:residual_efficiency} shows that the equilibrated flux estimator of Theorem~\ref{thm:upper_bound} is locally efficient and robust.
\begin{theorem}[Local efficiency and robustness]\label{thm:efficiency_robustness}
Let $u$ be the weak solution of problem~\eqref{eq_RD} given by~\eqref{eq:RD_weak} and let $\uh \in \Vh $ be its finite element approximation given by~\eqref{eq:RD_fem}.
Let $\sh \in \RTNp \cap \Hdv$ and $\phih \in \PP_p(\Th)$ be given by Definition~\ref{def:recs}. Then, for each mesh element $\elm\in\Th$, there holds
\begin{equation}\label{eq:robustness_efficiency_flux}
\wk^2 \norm{\eps \nabla \uh + \eps^{-1}\sh}_{\elm}^2 + \norm{\kappa\left(\uh - \phih\right)}_{\elm}^2
\lesssim \sum_{\elmt\in\frakT_\elm} \left[ \NORM{u-\uh}_{\elmt}^2 + \alpha_{\elmt}^2 \norm{f-\Pih f}_{\elmt}^2 \right],
\end{equation}
where the constant in $\lesssim$ depends only on the dimension $\dim$, the shape-regularity constant $\sreg$ of $\Th$, and on the polynomial degree $p$, so that it is independent of the parameters $\eps$ and $\kappa$ and the mesh-sizes $h_\elm$.
\end{theorem}

\section{Necessity of the weights $\wk$ in the upper bound}\label{sec_countr_ex}

Theorems~\ref{thm:upper_bound} and \ref{thm:efficiency_robustness} show that the estimator $\wk \norm{\eps \nabla \uh + \eps^{-1}\sh}_{\elm} + \norm{\kappa\left(\uh - \phih\right)}_{\elm}$ obtained from the flux equilibration of Definition~\ref{def:recs} is a reliable, locally efficient, and robust energy error estimator for singularly perturbed reaction--diffusion problems.
Here we show the {\em necessity} of the weight $\wk$ for robustness of equilibrated flux estimators that involve only piecewise polynomial vector fields on $\Th$. We also recall that an alternative option, related to the approach in~\cite{Ains_Bab_rel_rob_a_post_RD_99,Ainsw_Vej_guar_rob_RD_11,Ainsw_Vej_guar_rob_RD_14,Kopt_a_post_equi_fl_anis_RD_17}, is to perform an equilibrations on a submesh.

\subsection{Necessity of the weights $\wk$}

The following proposition applies to any flux equilibration on $\Th$:

\begin{proposition}[Best-possible bound by piecewise polynomials of the mesh $\Th$]\label{prop:flux_lower_bound}
Let $\uh \in \PP_p(\Th)\cap \Hoo$ be an arbitrary piecewise $p$-degree polynomial, $p \geq 1$, and let the face residual term $j_\idx$ be defined by~\eqref{eq:face_res}.
Let $\PP_{p^\prime}(\Th;\R^{\dim})$ denote the space of $\R^{\dim}$-valued piecewise polynomials of degree at most $p^\prime$ over $\Th$, where $p^\prime\geq 0$ is an arbitrary nonnegative integer.
Then,
\begin{equation}\label{eq:lower_bound_flux}
\inf_{\bvh \in \Hdv\cap \PP_{p^\prime}(\Th;\R^{\dim})  } \norm{ \eps \nabla \uh + \eps^{-1}\bvh} \gtrsim \sqrt{\frac{\kappa \underline{h}}{\eps}} \left( \sum_{F\in\calFint} \eps^{-1}\af \norm{j_\idx}_F^2 \right)^{\frac{1}{2}},
\end{equation}
where $\underline{h}\eq\min_{\elm\in\Th} h_K$, and where the constant depends only on the polynomial degrees $p$ and $p^\prime$, the dimension $\dim$, and the shape-regularity $\sreg$ of $\Th$.
\end{proposition}

\begin{proof}
Let $\bvh\in\Hdv\cap \PP_{p^\prime}(\Th;\R^{\dim})$ be arbitrary. Then, for each interior face $F\in\calFint$, the $\Hdv$-conformity of $\bvh$ implies that $\jump{ \bvh \scp \bn_F }_F=0$, and hence $j_\idx|_F = - \eps \jump{ (\eps\nabla \uh + \eps^{-1}\bvh)\scp \bn_F }_F$. Since $(\eps\nabla \uh + \eps^{-1}\bvh)|_{K}\in \PP_{\max(p^\prime,p-1)}(\elm;\R^\dim)$ for each element $K\in\Th$, we can apply the triangle inequality and the inverse inequality (analogous to~\eqref{eq:inverse_inequality}) to find that, for any $F\in\calFint$,
\begin{equation}\label{eq:lower_bound_flux_2}
\eps^{-1} \af  \norm{j_\idx}^2_F \lesssim \frac{\eps}{\kappa \underline{h}}\sum_{\elm\in\Th,F\subset\p\elm} \norm{\eps \nabla \uh + \eps^{-1}\bvh}_{\elm}^2.
\end{equation}
Therefore, we get~\eqref{eq:lower_bound_flux} by summing~\eqref{eq:lower_bound_flux_2} over all faces $F\in\calFint$, and recalling that $\bvh$ was arbitrary.
\end{proof}

The upshot of Proposition~\ref{prop:flux_lower_bound} is that for any problem where the jump estimators are sufficiently dominant, \ie\ when
\begin{equation}\label{eq:jump_dominance}
\NORM{u-\uh} \simeq \left( \sum_{F\in\calFint} \eps^{-1}\af \norm{j_\idx}_F^2 \right)^{\frac{1}{2}},
\end{equation}
then \emph{any} error estimator involving a term of the form $\norm{ \eps \nabla \uh + \eps^{-1}\bvh}$ without any weight will necessarily be non-robust when $\kappa \underline{h} / \eps$ takes large values, since~\eqref{eq:lower_bound_flux} and~\eqref{eq:jump_dominance} then imply
\begin{equation}\label{eq:inefficiency}
\inf_{\bvh \in \Hdv\cap \PP_{p^\prime}(\Th;\R^{\dim})  } \frac{\norm{ \eps \nabla \uh + \eps^{-1}\bvh}}{\NORM{u-\uh}} \gtrsim \sqrt{\frac{\kappa \underline{h} }{\eps}}.
\end{equation}
In other words, the effectivity index can be become {\em arbitrarily large} in the singularly-perturbed regime when the weight $\wk$ is not included. It is then seen that the inclusion of the weight term $\wk$ in Theorem~\ref{thm:upper_bound} is {\em necessary} when considering flux equilibrations from vector-valued piecewise polynomial subspaces of $\Hdv$ on the mesh $\Th$, regardless of the precise details of the construction of the flux. Examples of flux equilibrations proposed in the past that cannot be robust in general include Repin and Sauter~\cite{Rep_Sau_funct_a_post_react_dif_06}, Ainsworth~\eal\ \cite{Ains_All_Bar_Rank_a_post_ADR_13}, Eigel and Samrowski~\cite{Eig_Sam_func_a_post_RD_14}, Eigel and Merdon~\cite{Eig_Mer_equil_ADR_16}, and~Vejchodsk\'y~\cite{Vej_compl_a_post_12,Vej_qual_equil_15,Vej_RD_ENUMATH_17}.

We now present an example of a situation where~\eqref{eq:jump_dominance} holds and where $\kappa \underline{h}/\eps$ can be arbitrarily large. In fact the example is similar to the one in~{\upshape\cite[Section~2.3]{Ains_Bab_rel_rob_a_post_RD_99}}, albeit with some suitable adjustments.

\begin{example}[Dominant jump estimators]
Let $\Om \eq (-1/2,1/2)$ and let $m$ be an odd integer that will later on be chosen sufficiently large. Consider a uniform mesh $\Th$ of $\Om$ with $2N = (m+1)^2$ intervals, $N \eq (m+1)^2/2$, and mesh size $h \eq 1 / (2N) = 1 / (m+1)^2$. Hence, the interior nodes are $x_i = ih$, where $i\in\{-N+1,\dots, N-1\}$.
Let
\[
    f \eq \calI_\idx \cos(m\pi x) \in \PP_1(\Th) \cap \Hoo
\]
denote the piecewise affine Lagrange interpolant (preserving the point values) of the function $x\mapsto \cos(m\pi x)$; it follows from the fact that $m$ is odd that $f\in \Hoo$. Note that in the example of {\upshape\cite{Ains_Bab_rel_rob_a_post_RD_99}}, the function $f$ was chosen as $\cos(\pi x)$ instead.

Consider now problem~\eqref{eq_RD} along with its finite element approximation~\eqref{eq:RD_fem} in the space $\Vh = \PP_1(\Th)\cap \Hoo$. It is easy to show that
\[
    \uh = (\eps^2 \mu_h + \kappa^2)^{-1} f
\]
is the discrete solution, where
\[
    \mu_h\eq \frac{6}{2+\cos(m \pi h)} \frac{1-\cos(m\pi h)}{h^2},
\]
as a result of the identity
\[
    \int_{-1/2}^{1/2} f^\prime \vh^\prime \,\dd x = \mu_h \int_{-1/2}^{1/2}
    f \vh\,\dd x \qquad \forall \vh\in \Vh.
\]
Then, noting that interior vertices and faces coincide for problems in one space dimension, it is found that
\[
r_\idx|_{\elm}= \frac{\eps^2 \mu_h}{\eps^2\mu_h + \kappa^2} f|_{\elm} ,\quad j_\idx|_{x_i}=-\eps^2\jump{ \uh^{\prime}(x_i) } =  \frac{\eps^2}{\eps^2\mu_h + \kappa^2} \frac{2(1-\cos(m\pi h))}{h} f(x_i).
\]

Now, since $\lim_{m\tends\infty} \frac{1-\cos(m\pi h)}{h} = \frac{\pi^2}{2}$ when $h=h(m)=1/(m+1)^2$, we can pick $m$ sufficiently large such that $\mu_h\simeq h^{-1}$. Suppose also henceforth that $\kappa h / \eps \geq 1$, so that $\alpha_\elm$ given by~\eqref{eq:residual_weights} takes the value $1/\kappa$. Then, we find that
\[
\sum_{K\in\Th} \ak^2 \norm{r_\idx}_{\elm}^2 = \frac{1}{\kappa^2}\left(\frac{\eps^2 }{\eps^2\mu_h + \kappa^2} \right)^2 \mu_h^2 \norm{f}^2 \simeq \left(\frac{\eps^2}{\eps^2\mu_h + \kappa^2} \right)^2 \frac{1}{\kappa^2 h^2}.
\]
We also obtain
\[
\sum_{F\in\calFint}\eps^{-1}\af \norm{j_\idx}_{F}^2 \simeq \left(\frac{\eps^2}{\eps^2\mu_h + \kappa^2}\right)^2 \frac{1}{\eps\kappa} \sum_{i=-N+1}^{N-1}\abs{f(x_i)}^2 \simeq  \left(\frac{\eps^2}{\eps^2\mu_h     + \kappa^2}\right)^2 \frac{1}{\eps\kappa h},
\]
where we have used the trigonometric identity
\[
    \sum_{i=-N+1}^{N-1}\abs{f(x_i)}^2 = \sum_{i=-N+1}^{N-1} \Big\lvert\cos\Big(\frac{m\pi i}{2N}\Big)\Big\rvert^2 =\sum_{i=-N+1}^{N-1}\Big\lvert\cos\Big(\frac{(\sqrt{2N} - 1)\pi i}{2N}\Big)\Big\rvert^2 = N  = \frac{1}{2h}.
\]
Since $\eps \kappa h \leq \kappa^2 h^2$, we see that
\[
\sum_{K\in\Th} \ak^2 \norm{r_\idx}_{\elm}^2 \lesssim \sum_{F\in\calFint}\eps^{-1}\af \norm{j_\idx}_{F}^2  \iff \NORM{u-\uh}^2 \simeq \sum_{F\in\calFint}\eps^{-1}\af \norm{j_\idx}_{F}^2,
\]
where we note that there is no data oscillation since $f\in\PP_1(\Th)$. Hence this provides an example where~\eqref{eq:jump_dominance} holds, and the factor $\kappa h/\eps$ can be made arbitrarily large.
\end{example}

\subsection{Flux equilibration on a submesh}

We finish with the following remark:

\begin{remark}[Flux equilibration on boundary-layer adapted submeshes]
The approach in~{\upshape\cite{Ainsw_Vej_guar_rob_RD_11,Ainsw_Vej_guar_rob_RD_14,Kopt_a_post_equi_fl_anis_RD_17}, following {\upshape\cite{Ains_Bab_rel_rob_a_post_RD_99}}}, can be seen as defining a flux $\bssub\in\Hdv$ that satisfies an equilibration property similar to~\eqref{eq:equilibration}, yet with the key difference that $\bssub$ is defined with respect to a submesh $\widetilde{\Th}$ of $\Th$ with thin elements that are adapted to the parameters $\eps$ and $\kappa$ and local mesh-size (see \eg\ {\upshape\cite[Fig.~3]{Ainsw_Vej_guar_rob_RD_11}}).
In this case, the argument in the proof of Proposition~\ref{prop:flux_lower_bound} does not apply, because the inverse inequality $\norm{(\eps\nabla \uh + \eps^{-1}\bssub)\scp\tn_F}_{F} \lesssim h_\elm^\mft \norm{\eps\nabla \uh + \eps^{-1}\bssub}_{\elm}$, $F \subset \p \elm$, $\elm \in \Th$, is not applicable when $\bssub \in \PP_{p^\prime}(\widetilde{\Th};\R^{\dim})$ but $\bssub\notin \PP_{p^\prime}(\Th;\R^{\dim})$.
This essentially shows how there are now two different approaches to constructing robust equilibrated flux estimators. Either the flux is computed as a piecewise polynomial vector field with respect to the original mesh, in which case the inclusion of a weight of the form of $\wk$ from~\eqref{eq_weights} in the upper bound is necessary, or one constructs the flux with respect to some other sufficiently rich subspace of $\Hdv$, such as a piecewise polynomial subspace with respect to an adapted submesh $\widetilde{\Th}$ of $\Th$, in which case the weights are not necessary.
\end{remark}

\appendix
\section{Explicit constants for the inverse inequality}\label{sec:appendix}

For each polynomial degree $p\geq 0$, let $\Cpo$ denote the best constant of the inverse inequality for the unit interval $(0,1)$, i.e.\
\begin{equation}\label{eq:Cpo}
\begin{aligned}
\norm{v^\prime}_{L^2(0,1)} \leq \Cpo \norm{v}_{L^2(0,1)} &&&\forall\, v\in \PP_p(0,1),
\end{aligned}
\end{equation}
where $\PP_p(0,1)$ denotes the space of univariate polynomials of degree at most $p$ on $(0,1)$.
It was shown in \cite{Koutschan_Neumuller_16} that, for all $p\geq 0$,
\begin{equation}\label{eq:Cpo_bound}
\Cpo \leq \frac{1}{\sqrt{2}} \sqrt{(p-1) p (p+1) (p+2)},
\end{equation}
where we have taken into account the fact that we consider $\Cpo$ on the unit interval $(0,1)$ rather than the interval $(-1,1)$ as in \cite{Koutschan_Neumuller_16}. This improves on earlier bounds, e.g.\ in \cite{Schwab_98}.

We will show here explicit bounds for the constants of the inverse inequality for hypercubes and simplices in terms of $\Cpo$.

\subsection{Unit hypercube}

For an integer $\dim\geq 1$, let $\sD$ be a shorthand notation for $\{1,\dots,\dim\}$.
Let $Q_{\dim} \coloneqq\{ x\in \R^\dim,\; \abs{x}_{\infty}\leq 1, \;x_i \geq 0 \;\;\forall i\in\sD \}$ denote the unit hypercube in $\R^\dim$, where $\abs{x}_{\infty}\coloneqq\max_{i\in\sD}\abs{x_i}$. Let $\PP_p(Q_{\dim})$ denote the space of polynomials of total degree at most $p$ on $Q_{\dim}$.

\begin{lemma}\label{lem:const_hypercube}
For all $\dim\geq 1$ and all $p\geq 0$, we have
\begin{equation}\label{eq:hypercube_bound}
\begin{aligned}
\norm{ v_{x_i} }_{L^2(Q_{\dim})} \leq \Cpo \norm{v}_{L^2(Q_{\dim})} &&&\forall\, v\in \PP_p(Q_{\dim}),\quad \forall i\in\sD.
\end{aligned}
\end{equation}
\end{lemma}
\begin{proof}
After a possible re-labelling of the indices, it is enough to show that~\eqref{eq:hypercube_bound} holds for the case $i=1$. Then, writing $x=(x_1,x^\prime)$ with $x^\prime\in \R^{\dim-1}$, we see that
\[
\int_{\Qd} \abs{v_{x_1}}^2 \dd x = \int_{Q^{\dim-1}} \int_{0}^1 \abs{v_{x_1}(x_1,x^\prime)}^2 \dd x_1 \dd x^\prime \leq \Cpo^2 \int_{Q^{\dim-1}} \int_{0}^1 \abs{v(x_1,x^\prime)}^2\dd x_1 \dd x^{\prime} = \Cpo^2 \int_{\Qd} \abs{v}^2 \dd x,
\]
where we use the fact that $x_1\mapsto v(x_1,x^\prime)$ is in $\PP_p(0,1)$ for all $x^\prime$.
\end{proof}

\subsection{Unit simplex}

For a parameter $t>0$, let $K^{\dim}_{t}\coloneqq \{ x\in \R^\dim,\; \abs{x}_{1}\leq t, \;x_i \geq 0 \;\;\forall i\in\sD \}$, where $\abs{x}_1 \eq \sum_{i=1}^\dim \abs{x_i}$, denote the simplex in $\R^\dim$ with side-length $t$. If $t=1$, we adopt the simpler notation $\Kd\coloneqq K^{\dim}_1$.
Let $\Cpd$ denote the best constant such that
\begin{equation}\label{eq:simplex_inverse_inequality}
\begin{aligned}
\norm{ v_{x_i} }_{L^2(\Kd)} \leq \Cpd \norm{v}_{L^2(\Kd)} &&&\forall\, v\in \PP_p(\Kd),\quad \forall i\in\sD.
\end{aligned}
\end{equation}
We shall obtain here an explicit bound for the constant $\Cpd$ in terms of the space dimension $\dim$ and the constant $\Cpo$ of~\eqref{eq:Cpo}.

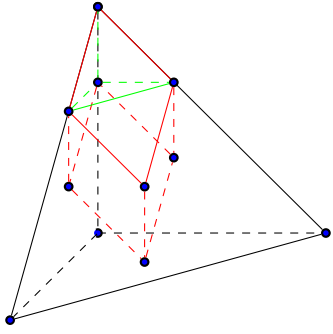
\begin{figure}[tbh]
\begin{center}
  \begin{tikzpicture}[scale=3,st1/.style={circle,draw=black,fill=blue,thick, minimum size=1.0mm, inner sep=0pt}]
  \node[st1] (A) at (1,0,0) {};
  \node[st1] (B) at (0,1,0) {};
  \node[st1] (C) at (0,0,1) {};
  \draw[dashed] (0,0) node[st1] {} -- (A);
  \draw (A)-- (B) -- (C) -- (A);
  \draw[dashed] (B) -- (0,0) -- (C);
  \node[st1] (a) at (0,1,0) {};
  \node[st1] (b) at (0,0.666,0.333) {};
  \node[st1] (c) at (0.3333,0.333,0.333) {};
  \node[st1] (d) at (0.333,00,0.333) {};
  \node[st1] (e) at (0,0.3333,0.333) {};
  \node[st1] (f) at (0,0.666,0.0) {};
  \node[st1] (g) at (0.333,0.333,0.0) {};
  \node[st1] (h) at (0.333,0.666,0.0) {};
  \draw[red] (a) -- (b) -- (c) -- (h) -- (a);
  \draw[red,dashed] (h) -- (g) -- (f) -- (e) -- (b) ;
  \draw[red,dashed] (e) -- (d) -- (g);
  \draw[red,dashed] (c) -- (d) ;
  \draw[green,dashed] (b)--(f)--(h) ;
  \draw[green,dashed] (f)--(a);
  \draw[green] (h)--(b);

\end{tikzpicture}
\caption{Subdivision of the unit simplex used in the proof of Theorem~\ref{thm:const_unit_simplex}. The unit simplex is shown for $\dim=3$, along with its sub-simplex $K_{\dagger}$ (edges shown in green) and sub-parallelepiped $Q_{\dagger}$ (edges shown in red).}
\label{fig:simplex_subdivision}
\end{center}
\end{figure}

\begin{theorem}\label{thm:const_unit_simplex}
For all $\dim \geq 1$ and for all $p\geq 0$, the best constant $\Cpd$ in \eqref{eq:simplex_inverse_inequality} satisfies
\begin{equation}\label{eq:const_unit_simplex}
\begin{aligned}
\Cpd \leq \frac{\sqrt{5}}{4}\, (2\sqrt{2})^{\dim} \Cpo.
\end{aligned}
\end{equation}
\end{theorem}
\begin{proof}
The proof is based on an induction on the dimension, where we seek to bound $\Cpd$ in terms of $C_{p,\dim-1}$, $\Cpo$, and $\dim$.
Without loss of generality, it is enough to consider only the case $i=1$ in~\eqref{eq:simplex_inverse_inequality}, after a possible re-labelling of the indices. Then, writing $x=(x^\prime,x_d)$ with $x^\prime\in \R^{d-1}$, we have
$\int_{\Kd}\abs{v_{x_1}}^2 \dd x = \int_0^1 \int_{K^{\dim-1}_{1-x_{\dim}}} \abs{v_{x_1}}^2 \dd x^\prime \dd x_{\dim}$.
Since, for fixed $x_{\dim}\in (0,1)$, $x^\prime\mapsto v(x^\prime,x_d)$ is a polynomial of degree at most $p$ on $K^{\dim-1}_{1-x_{\dim}}$, it would be natural to apply the inverse inequality for simplices of dimension $d-1$ after a suitable scaling. However, a difficulty arises for $x_{\dim}$ close to $1$ due to the appearance of a negative power of $1-x_{\dim}$ inside the resulting integral. We can overcome this obstacle using an appropriate subdivision of the unit simplex and a change of variables.

The proof proceeds in two steps. We first treat the case $\dim=2$ and show that~\eqref{eq:const_unit_simplex} holds (we actually consider $\dim \geq 2$ below for the sake of generality), and then the induction is carried out on $\dim$ with a different argument, leading to a sharper bound than that would result from step~1 only.

{\em Step~1.} Let $d \geq 2$ and consider the partition of $K$ into $K_* \coloneqq \{ x\in K, x_{\dim} < 1-1/\dim \}$ and $K_{\dagger} \coloneqq K\setminus K_*$. Then, $\norm{v_{x_1}}_{L^2(\Kd)}^2=\norm{v_{x_1}}_{L^2(K_*)}^2+\norm{v_{x_1}}_{L^2(K_\dagger)}^2$, and the first term can be bounded as follows:
\begin{multline}\label{eq:const_unit_simplex_1}
 \norm{v_{x_1}}_{L^2(K_*)}^2 = \int_0^{1-1/d} \left(\int_{K^{\dim-1}_{1-x_{\dim}}} \abs{v_{x_1}}^2 \dd x^\prime \right)\dd x_{\dim}
\\ \leq  \int_0^{1-1/d} \left(\frac{C_{p,\dim-1}^2}{(1-x_d)^2} \int_{K^{\dim-1}_{1-x_{\dim}}} \abs{v}^2 \dd x^\prime \right)\dd x_{\dim}
 \leq d^2 C_{p,\dim -1}^{2} \norm{v}_{L^2(K_*)}^2,
 \end{multline}
where crucially we use the fact that $(1-x_d)^{-2}\leq d^2$ for $x_d\leq 1-1/d$. In order to bound the second term $\norm{v_{x_1}}^2_{L^2(K_{\dagger})}$, we introduce a change of coordinates in terms of the affine  map $F$ defined by
\[
F(\xi) \coloneqq e_{\dim} + \sum_{i=1}^{\dim} (e_{i-1} - e_{\dim}) \xi_i,
\]
where $e_0=0$, and $e_i$ is the $i$-th unit vector for $1\leq i  \leq \dim$. Letting $x=F(\xi)$, we have $x_{j} = \xi_{j+1}$ for $j\leq \dim -1$, and $x_{\dim} = 1-\sum_{i=1}^{d}\xi_i$. The inverse is then given by $\xi_1 = 1- \sum_{i=1}^d x_i$, and $\xi_{j}=x_{j-1}$ for $2\leq j \leq \dim$.
It is thus easily seen that $F$ is a bijection from $K$ onto itself, and that $F(0)=x_d$. Thus $F$ corresponds to a change of coordinates on the unit simplex.
Additionally, it can be shown that the Jacobian $\abs{\mathrm{det}DF}=1$.

Let $Q^{\dim}_{1/d} \coloneqq \{ \xi \in Q^{\dim}, \, \abs{\xi}_{\infty}\leq 1/d\}$ be a hypercube with side length $1/d$, and let $Q_{\dagger}$ be the parallelepiped obtained as the image of $Q^{\dim}_{1/\dim}$ under the mapping $F$, i.e.\ $Q_{\dagger} = F(Q^{\dim}_{1/d})$.
It is then easy, but tedious, to show that
 \begin{equation}\label{eq:F_mapping_inclusions}
 K_{\dagger} \subset Q_{\dagger}  \subset K.
 \end{equation}
Figure~\ref{fig:simplex_subdivision} illustrates the sets $K_{\dagger}$, $Q_{\dagger}$, and $K$ for the case $\dim=3$.
 Now, let $\widetilde{v}(\xi) = v(F(\xi))$ be the pullback of $v$ under $F$. Since $F$ is affine, $\widetilde{v} \in \PP_p(\Kd)$. It is also easy to check that $v_{x_1} = \widetilde{v}_{\xi_2} - \widetilde{v}_{\xi_1}$.
 Using the change of variables and the fact that $\abs{\mathrm{det} DF}=1$, it follows from \eqref{eq:F_mapping_inclusions} that
 \[
 \norm{ v_{x_1}}_{L^2(K_{\dagger})}^2 \leq  \norm{ v_{x_1}}_{L^2(Q_{\dagger})}^2  = \norm{\widetilde{v}_{\xi_2} - \widetilde{v}_{\xi_1}}_{L^2(Q^{\dim}_{1/d})}^2 \leq 2 (\norm{\widetilde{v}_{\xi_2}}_{L^2(Q^{\dim}_{1/d})}^2 + \norm{\widetilde{v}_{\xi_1}}_{L^2(Q^{\dim}_{1/d})}^2).
 \]
Applying the inverse inequality for hypercubes, namely $\norm{\widetilde{v}_{\xi_i}}_{L^2(Q^{\dim}_{1/d})}^2 \leq d^2 \Cpo^2 \norm{\widetilde{v}}_{L^2(Q^{\dim}_{1/\dim})}^2$, and changing back to the original variables, we then obtain from the second inclusion in \eqref{eq:F_mapping_inclusions} that
\begin{equation}\label{eq:const_unit_simplex_2}
\norm{v_{x_1}}^2_{L^2(K_{\dagger})} \leq 4 d^2 \Cpo^2 \norm{v}^2_{L^2(Q_{\dagger})} \leq 4 d^2 \Cpo^2 \norm{v}_{L^2(\Kd)}^2.
 \end{equation}
 Therefore, combining \eqref{eq:const_unit_simplex_1} and \eqref{eq:const_unit_simplex_2}, we arrive at $ \norm{v_{x_1}}^2_{L^2(\Kd)} \leq d^2( C_{p,\dim-1}^2 + 4 \Cpo^2 ) \norm{v}_{L^2(\Kd)}^2$, for any $v\in \PP_p(\Kd)$. This implies $
 \Cpd^2 \leq d^2( C_{p,\dim-1}^2 + 4 \Cpo^2 )$, and thus~\eqref{eq:Cpo} and an induction argument show that
\begin{equation}\label{eq:const_unit_simplex_3}
\Cpd \leq \textstyle{\left(1+  4 \sum_{j=1}^{\dim-1} \frac{1}{(j!)^2} \right)^{\frac{1}{2}}}\, d!\, \Cpo
\end{equation}
for any $d \geq 2$. This shows~\eqref{eq:simplex_inverse_inequality}, but with a worse constant than that of~\eqref{eq:const_unit_simplex} for $\dim \geq 3$. For this reason, we proceed in a second step in a different way.

{\em Step~2.} Let $\dim\geq 3$. We again subdivide the simplex $K$, this time as
\[
\begin{aligned}
K=  \{x\in K,\, x_d \leq 1/2\}\cup \{x\in K,\, x_{d-1}\leq 1/2\}.
\end{aligned}
\]
Furthermore, for any fixed $x_{\dim-1}$, $x_{\dim-1}^\prime=(x_1,\dots,x_{d-2},x_{d})^\prime\mapsto v(x)$ is a polynomial of degree at most $p$ on a simplex that is isometric to $K^{\dim-1}_{1-x_{\dim-1}}$. Let also $x_{\dim}^\prime=(x_1,\dots,x_{\dim-1})$. Crucially, since $\dim\geq 3$ and we subdivide above into two subsets, we can avoid the critical subset $K_{\dagger}$ of Step~1 as
\begin{multline}
\norm{v_{x_1}}_{L^2(\Kd)}^2 \leq
 \sum_{j=\dim-1}^{\dim} \int_{0}^{1/2}\left(\int_{K^{ \dim-1}_{1-x_{j}}} \abs{v_{x_1}}^2\dd x_{j}^\prime\right)\dd x_j
\\ \leq  \sum_{j=\dim-1}^\dim \int_{0}^{1/2} \frac{C_{p,\dim-1}^2}{(1-x_{j})^2} \left( \int_{K^{\dim-1}_{1-x_j}} \abs{v}^2 \dd x_j^\prime \right)\dd x_j\leq 8 C_{p,\dim-1}^2 \norm{v}_{\Kd}^2.
\end{multline}
It then follows by induction that $\Cpd \leq (2\sqrt{2})^{\dim-2} C_{p,2}$ for all $\dim \geq 3$. Since $C_{p,2}\leq 2\sqrt{5}\, C_{p,1}$ by \eqref{eq:const_unit_simplex_3}, we get~\eqref{eq:const_unit_simplex}.
\end{proof}
 Applying Theorem~\ref{thm:const_unit_simplex} to the cases $\dim=2$ and $\dim=3$ gives the following explicit bounds
 \begin{equation}
   \begin{aligned}
   C_{p,2} & \leq \sqrt{10 (p-1) p (p+1) (p+2)}, &&&  C_{p,3} & \leq  \sqrt{80 (p-1) p (p+1) (p+2) }.
   \end{aligned}
 \end{equation}

\subsection{General simplex}

Let $K$ be a simplex in $\R^\dim$, $\dim\geq 2$, and let $\hat{K}$ denote the unit simplex.
 Let $J_K$ denote the differential of the affine transformation mapping  $T_K\colon\hat{K}\tends K$.
For $\bm{v}\in \RTNK$, we define the Piola transformation $\bm{\hat{v}} \in \bm{RTN}_p(\hat{K})$ by
\begin{equation}\label{eq:Piola}
\bm{\hat{v}}(\hat{x}) = \abs{\det J_K} J_K^{-1}[ \bm{v}\circ T_K(\hat{x}) ].
\end{equation}

\begin{lemma}[Ciarlet~\cite{Ciar_78} Thm 3.1.2 and~\cite{Di_Pietr_Ern_book_12}] \label{lem:geometric_bounds}
 There holds
\begin{equation}\label{eq:geometric_bounds}
\begin{aligned}
\norm{J_K}_2 \leq \frac{h_K}{{\rho_{\hat K}}}, &&& \norm{J_K^{-1}}_2 \leq \frac{\sqrt{2}}{\rho_K},&&& \abs{\det J_K} = \frac{\abs{K}_d}{\abs{\hat{K}}_d}, &&& \frac{\abs{\p K}_{\dim-1}}{\abs{K}_{\dim}} \leq (d+1)d \,\sreg\, h_K^{-1}.
\end{aligned}
\end{equation}
\end{lemma}
Note that in~Lemma~\ref{lem:geometric_bounds}, we have used the fact that the diameter of the unit simplex is $\sqrt{2}$ for all $\dim\geq 2$.

\begin{lemma}[Hesthaven \& Warburton \cite{HesthavenWarburton2003}]
  Let $v \in \PP_{p}(K)$. Then
  \[
  \norm{v}_{\p K} \leq \sqrt{\frac{(p+1)(p+\dim)}{\dim} \frac{\abs{\p K}_{\dim-1}}{\abs{K}_{\dim}} } \norm{v}_K.
  \]
\end{lemma}
Therefore, for $\bm{v}\in \RTNK$, we have
\[
\begin{aligned}
h_K^{1/2}\norm{\bm{v}\cdot \bm{n}}_{\p K} \leq \Cpdp \norm{\bm{v}}_K, &&&
\Cpdp \coloneqq \sqrt{(d+1)(p+2)(p+d+1) \sreg}.
\end{aligned}
\]

\begin{lemma}
Let $K$ be a simplex in $\R^{\dim}$ and $\bm{v}\in\RTNK$. Then,
  \begin{equation}\label{eq:divergence_constant_inverse_inequality}
  \begin{aligned}
  h_K \norm{\nabla \cdot \bm{v}}_K \leq \CpdK \norm{\bm{v}}_K, &&& \CpdK \coloneqq \sqrt{2d}\, \sreg \, C_{p+1,\dim},
  \end{aligned}
  \end{equation}
  where $\Cpd$ is characterized in Theorem~\ref{thm:const_unit_simplex}.
\end{lemma}
\begin{proof}
Using the Piola Transformation, we have $ \nabla \cdot \bm{v}=\nabla_{\hat{x}}\cdot \bm{\hat{v}} / \abs{\det J_K}$, therefore,
$\norm{\nabla \cdot\bm{v}}_K^2 \leq \abs{\det J_K}^{-1} \norm{\nabla\cdot \bm{\hat{v}}}_{\hat{K}}^2$. Then, since $\bm{\hat{v}}_{i} \in \PP_{p+1}(\hat{K})$ for $i\in \{1{:}\dim\}$, we apply Theorem~\ref{thm:const_unit_simplex} to obtain
$\norm{\nabla\cdot \bm{\hat{v}}}_{\hat{K}}^2 \leq d \sum_{i=1}^{\dim} \norm{\bm{\hat{v}}_{i,x_i}}_{\hat{K}}^2 \leq d C_{p+1,\dim}^2 \norm{\bm{\hat{v}}}_{\hat{K}}^2$.
Then, using the definition of the Piola transformation, it is seen that $\norm{\bm{\hat{v}}}_{\hat{K}}^2\leq \norm{J_K^{-1}}_{2}^2 \abs{\det J_K}\norm{\bm{v}}_K^2$. We then use the bound $\norm{J_K^{-1}}_{2} \leq \sqrt{2} \sreg h_K^{-1}$ from~\eqref{eq:geometric_bounds} to find that $h_K\norm{\nabla\cdot\bm{v}}_K \leq \sqrt{2d} \sreg C_{p+1,\dim}\norm{\bm{v}}_K$, which finishes the proof.
\end{proof}

\def\polhk#1{\setbox0=\hbox{#1}{\ooalign{\hidewidth
  \lower1.5ex\hbox{`}\hidewidth\crcr\unhbox0}}}
  \def\polhk#1{\setbox0=\hbox{#1}{\ooalign{\hidewidth
  \lower1.5ex\hbox{`}\hidewidth\crcr\unhbox0}}} \def\cprime{$'$}

\end{document}